\documentclass[11pt]{article}

\makeatletter
\providecommand{\@msc}{}
\makeatother
\usepackage[margin=1in]{geometry}

\usepackage[
  backend=biber,
  style=alphabetic,
  maxalphanames=4,
  minalphanames=4,
  maxbibnames=99
]{biblatex}

\usepackage{amsmath,amssymb,amsthm}
\usepackage{hyperref}
\usepackage[pdftex,dvipsnames]{xcolor}
\definecolor{rptupink}{RGB}{209,56,150}

\hypersetup{
    colorlinks=true,
    citecolor=rptupink,
    linkcolor=rptudunkelblau,
    urlcolor=rptuviolett
}

\usepackage{amsmath,amssymb,amsthm}
\usepackage{hyperref}
\numberwithin{equation}{section} % Eqns numbered per section
\allowdisplaybreaks              % Allow multi-line displays to break across pages
\usepackage{mathrsfs}
\definecolor{rptu1}{RGB}{4,44,88}
\definecolor{rptu2}{RGB}{106,178,231}
\definecolor{rptublaugrau}{RGB}{80,114,137}
\definecolor{rptugruengrau}{RGB}{119,182,186}
\definecolor{rptudunkelblau}{RGB}{4,44,88}
\definecolor{rptuhellblau}{RGB}{106,178,231}
\definecolor{rptudunkelgruen}{RGB}{0,107,107}
\definecolor{rptuhellgruen}{RGB}{38,208,124}
\definecolor{rptuviolett}{RGB}{76,53,117}
\definecolor{rptupink}{RGB}{209,56,150}
\definecolor{rpturot}{RGB}{227,27,76}
\definecolor{rptuorange}{RGB}{255,162,82}
\definecolor{rptuschwarz}{RGB}{0,0,0}
\definecolor{rptuweiss}{RGB}{255,255,255}
\definecolor{darkpurple}{RGB}{75,0,130}

\usepackage{nameref}

\usepackage{cancel}

\title{Langevin-type dynamics and hypocoercivity on \texorpdfstring{$\mathrm{\textbf{SE}}(3)$}{SE(3)}}

\author{
  Martin Grothaus\\
  \texttt{grothaus@rptu.de}\\[1em]
  Andrea V. Hurtado-Quiceno\\
  \texttt{avanessa.hquiceno@edu.rptu.de}}

\usepackage{enumitem}

\def\A{A}
\def\L{L}
\def\I{I}
\def\S{S}
\def\B{B}
\def\Bor{\mathscr{B}}

\def\D{\mathcal{D}}

\def\SE{\textbf{$\mathrm{SE}$}}
\def\SO{\textbf{$\mathrm{SO}$}}

\def\P{\mathrm{\Pi}}
\def\trace{\mathrm{\trace}}

\def\Hi{H}

\usepackage{cancel}
\newtheorem{theorem}{Theorem}[section]
\newtheorem{lemma}[theorem]{Lemma}

\newtheorem{proposition}[theorem]{Proposition}
\newtheorem{corollary}[theorem]{Corollary}

\newtheorem{definition}[theorem]{Definition}

\newtheorem{remark}[theorem]{Remark}

\theoremstyle{definition}

\newtheorem{observation}[theorem]{Observation}

\newtheorem{notation}[theorem]{Notation}

\numberwithin{equation}{section}

\usepackage{xargs}                      % Use more than one 
\usepackage[colorinlistoftodos,prependcaption,textsize=tiny]{todonotes}
\newcommandx{\unsure}[2][1=]{\todo[linecolor=red,backgroundcolor=red!25,bordercolor=red,#1]{#2}}
\newcommandx{\change}[2][1=]{\todo[linecolor=blue,backgroundcolor=blue!25,bordercolor=blue,#1]{#2}}
\newcommandx{\info}[2][1=]{\todo[linecolor=OliveGreen,backgroundcolor=OliveGreen!25,bordercolor=OliveGreen,#1]{#2}}
\newcommandx{\improvement}[2][1=]{\todo[linecolor=Plum,backgroundcolor=Plum!25,bordercolor=Plum,#1]{#2}}
\newcommandx{\thiswillnotshow}[2][1=]{\todo[disable,#1]{#2}}

\newcommand{\AuthorInfo}[3]{%
  \noindent\textbf{#1}\\#2\\\textit{E-mail address:} \texttt{#3}\vspace{1em}\\
}

\date{\today}

\begin{document}

\maketitle

\begin{abstract}
\noindent
We construct hypocoercive Kolmogorov-type operators on the Euclidean motion group $\SE(3)=\mathbb{R}^3\rtimes\SO(3)$ using ZXZ Euler angles and study their long-time behaviour using the abstract hypocoercivity method. The operator $\L=\S-\A$ decomposes into a symmetric part $\S$, corresponding to a degenerate diffusion on the rotation 
component $\SO(3)$, and an antisymmetric part $\A$, which captures transport and torque effects induced by the Lie group structure. The associated invariant probability measure is $\mu_\Phi=\gamma_\Phi\otimes\nu_3$, where $\gamma_\Phi$ is a probability measure on $\mathbb{R}^3$ with density $Z(\Phi)^{-1}e^{-\Phi}$ with respect to the Lebesgue measure, 
and $\nu_3$ is the normalized Haar measure on $\SO(3)$.~\\

\noindent
As a consequence, if $\Phi\in C^\infty(\mathbb{R}^3)$ is bounded from below, $\gamma_\Phi$ satisfies a Poincar\'e inequality with constant 
$\Lambda>0$, and there exists $c<\infty$ such that
\begin{equation*}
    |\nabla_\xi^2\Phi(\xi)|
    \leq c\bigl(1+|\nabla_\xi\Phi(\xi)|\bigr),
    \quad \xi\in\mathbb{R}^3,
\end{equation*}
then the semigroup generated by the closure of $\L$ converges 
exponentially fast to equilibrium in $L^2(\SE(3),\mu_\Phi)$.~
%As a consequence, under suitable assumptions \improvement{Instead of suitable assumptions, just mention the assumptions. $\Phi \in C^{\infty}(\mathbb{R}^3)$} on the potential $\Phi$, the semigroup generated by the closure of $\L$ converges exponentially fast to equilibrium in $L^{2}(\SE(3),\mu_{\Phi})$.~
\end{abstract}

\begingroup
\renewcommand{\thefootnote}
\footnotetext{%

%\improvement{Discuss the Keywords and AMS.}
}
\endgroup

%----------- This show us the table of contents...
%\tableofcontents
%\bigskip

%%%%%%%%%%%%%%%%%%%%%%%%%%%%%%%%%%%%%%%%%%%%%%%%%%%%%%%%%%%%%%%%%%%%%%%%%%%%%%%%%%%%%%%%%%%%%%%%%%%%%%%%%%%%%%%%%%%%%%%%%
\section{Introduction}

In this article, we study hypocoercive Kolmogorov-type operators on the Lie group $\SE(3)$, the group of orientation-preserving Euclidean motions in  $\mathbb{R}^3$. This Lie group carries a natural sub-Riemannian structure, making it a canonical state space for degenerate diffusion processes arising in kinetic theory, geometric mechanics, and stochastic dynamics on manifolds.

%The term “hypocoercivity” was introduced by Cedric Villani in \cite{villani2009hypocoercivity} after Thierry Gallay suggested it to him, based on the analogy with the term “hypoellipticity”. 
In this setting, the dissipation operator acts only on a subset of variables, and additional structure is required to satisfy the Hörmander-type commutator relations and the abstract hypocoercivity method. 

In the geometric setting, however, the commutator structure required by Villani's approach is often not directly available. Grothaus and Klar observed this phenomenon in their study of kinetic equations on $\mathbb{R}^2 \times \mathbb{S}^1$ in \cite{GrotKlar}. This motivates the use of the abstract Hilbert space framework developed by Dolbeault, Mouhot and Schmeiser \cite{DMS15}, which is particularly well suited to geometric settings.  Grothaus and Stilgenbauer \cite{GROT-STIL-2014}, later translated the framework to the Kolmogorov backward setting, with particular emphasis on domain questions and essential m-dissipativity.

Beyond its intrinsic geometric interest, stochastic dynamics on Lie groups have recently attracted considerable attention in geometric mechanics and stochastic analysis. In particular, Langevin-type diffusions on matrix Lie groups provide natural models for rigid body motion subject to noise and damping, and are also increasingly relevant in geometric deep learning and optimization on manifolds. This connection motivates, in particular, the exploration of hypocoercivity as a tool for understanding convergence properties of stochastic gradient-based algorithms evolving on Lie groups and related homogeneous spaces. Recent developments include symplectic techniques for stochastic differential equations on reductive Lie groups and their associated Langevin dynamics; see, for example, Luesink and Street \cite{LUESINK2026114034}. These developments further motivate the study of convergence to equilibrium for degenerate diffusion processes on Lie groups, as well as potential applications of hypocoercive techniques to geometric optimization and learning algorithms.

The motivation for the present work is to study the long-time behaviour of solutions to a Langevin-type stochastic differential equation on the Euclidean motion group $\SE(3)=\mathbb{R}^3\rtimes \SO(3)$.~Let $(\xi_t,R_t)_{t\geq 0}$ be a stochastic process taking values in $\SE(3)$, where $\xi_t\in\mathbb{R}^3$ denotes the translation component and $R_t\in\SO(3)$ the rotation component. We consider the intrinsic Stratonovich SDE 

\begin{equation*}
\begin{aligned}
d\xi_t &= v_{3}(R_t) \,dt,\\
dR_t
&=
\frac{1}{2}\left(\nabla_{\xi} \Phi\left(\xi_t\right) \cdot v_{2}\left(R_t\right) \widetilde{X}_1\left(R_t\right)- \nabla_{\xi} \Phi\left(\xi_t\right) \cdot v_{1}\left(R_t\right) \widetilde{X}_2\left(R_t\right)\right) dt +\sigma\sum_{i=1}^{3}\widetilde{X}_i (R_t) \,dB_t^i,
\end{aligned}
\end{equation*}
where $\sigma \in (0, \infty)$, $\widetilde{X}_1, \widetilde{X}_2, \widetilde{X}_3$ are the right-invariant rotation vector fields on $\SO(3)$. The vectors $v_{i}(R) := Re_{i}$, where $R^{} \in \SO(3)$ is parametrized using the ZXZ Euler angles $\alpha, \beta, \gamma$, with $0 \leq \alpha, \gamma \leq 2\pi$ and $0 \leq \beta \leq \pi$ (see \cite[Section 9.4]{Chi-Alex}), applied to the canonical orthonormal basis $e_{i}, i = 1, 2, 3$, of $\mathbb{R}^3$ and $(B_t^1, B_t^2, B_t^3)$ is a standard Brownian motion in $\mathbb{R}^3$, see \cite{McKean1960}. Furthermore,  $\Phi \in C^\infty(\mathbb{R}^3)$ is a real valued potential such that there exists a finite normalization constant $$Z(\Phi):=\int_{\mathbb{R}^3}e^{-\Phi(\xi)} d\xi <\infty.$$

The diffusion acts only in the rotation directions, while the translational component is transported with the help of the antisymmetric operator $\A$.~ Consequently, the noise is degenerate, and the resulting generator is not coercive. Applying the Stratonovich-Itô correspondence yields the Kolmogorov operator $\L = \S - \A$, acting on smooth compactly supported functions $f \in \mathcal{D} := C_c^\infty(\SE(3))$, where the symmetric operator is given by
\begin{equation*}
\begin{aligned}
\S &=  \frac{\sigma^2}{2}\Delta_{\SO(3)} = -\frac{\sigma^2}{2}
\sum_{i=1}^{3}
\widetilde{X}_{i}^2,
\end{aligned}
\end{equation*}
and the antisymmetric operator is
\begin{equation*}
\A = - \widetilde{X}_6 -\frac{1}{2}
\Big((\nabla_{\xi}\Phi\cdot v_{2})\widetilde{X}_1  - (\nabla_{\xi}\Phi\cdot v_{1})\widetilde{X}_2 \Big).
\end{equation*}

The natural invariant probability measure on $\SE(3)$ is the product measure
\begin{equation*}
    \mu_{\Phi}:=\gamma_{\Phi}\otimes\nu_3,
\end{equation*}
where $d\gamma_{\Phi}(\xi):= Z(\Phi)^{-1}e^{-\Phi(\xi)}\,d\xi$ is the weighted probability measure on $\mathbb{R}^3$, and $\nu_3$ is the
normalized Haar measure on $\SO(3)$.~Thus,
$\mu_{\Phi}(\SE(3))=1$.~

This article demonstrates exponential convergence to equilibrium for the semigroup generated by the closure of $(\L,\mathcal{D})$.~Specifically, under appropriate conditions on the potential $\Phi: \mathbb{R}^3 \to \mathbb{R}$, including a Poincaré inequality for the measure $\gamma_\Phi$ and a growth condition on its Hessian, we confirm all requirements of the abstract hypocoercivity framework provided in \cite{GROT-STIL-2014}. Consequently, the closure of $(\L,\mathcal{D})$ generates a strongly continuous contraction semigroup $(T_t)_{t\geq0}$ on $L^2(\SE(3),\mu_\Phi)$, and there exist constants $0 < \kappa_1,\kappa_2 < \infty$ such that
\begin{equation*}
\|T_t f-(f,1)_{\Hi}\|_{\Hi} \leq
\kappa_1 e^{-\kappa_2 t}
\|f-(f,1)_{\Hi}\|_{\Hi}, \quad t\geq0.
\end{equation*}

This work builds on our recent construction of hypocoercive Langevin dynamics on the Euclidean motion group $\SE(2)$, see \cite{GHQ26}, and represents the first generalization of that framework. 

Moving from $\SE(2)$ to $\SE(3)$ introduces significant geometric complexity: unlike the abelian group $\SO(2) \simeq \mathbb{S}^1$, the rotation group $\SO(3)$ is non-commutative. This requires a more careful construction of the antisymmetric operator, which must generate the full Lie algebra $\mathfrak{se}(3)$ and remain compatible with the invariant probability measure $\mu_{\phi}$. In contrast, the symmetric operator involves only the rotation variables, i.e., $\S = \frac{\sigma^2}{2}\Delta_{\SO(3)}$.~Hence, the main challenge is to constructs an antisymmetric operator $A$ which is compatible with the invariant measure $\mu_\Phi$, verifies the corresponding macroscopic coercivity, and for which the necessary boundedness estimates within the abstract hypocoercivity method can be established.

Geometrically, this model extends the fiber lay-down dynamics on $\mathbb{R}^2 \times \mathbb{S}^1$ studied in \cite{GrotKlar} to the Lie group $\SE(3)$.~ While the fiber lay-down process represents orientation by a point on the sphere $\mathbb{S}^1$, our setting tracks the full orientation matrix in $\SO(3)$.~As a result, the state space is the Euclidean motion group $\SE(3) = \mathbb{R}^3 \rtimes \SO(3)$, resulting in a genuinely non-commutative geometric framework.

The paper is organized as follows. Section 2 reviews the geometric structure of $\SE(3)$ and summarizes the Hilbert space hypocoercivity method of \cite{GROT-STIL-2014}. In Section 3 we construct hypocoercive Kolmogorov operators on the Euclidean motion group $\SE(3)$ and verify them to satisfy all assumptions of the abstract hypocoercivity method. The main results of this article are as follows:

\begin{itemize}

   \item We construct a splitting $\L = \S - \A$ on $\mathcal{D} = C_c^\infty(\SE(3))$, where $\S$ is symmetric but degenerate and $\A$ is antisymmetric, both with respect to the probability measure $$\mu_{\Phi} = \gamma_\Phi\otimes\nu_3,$$
   such that Hörmander's condition is fulfilled.

   \item We proof essential m-dissipativity of $(L, \mathcal{D})$ in $L^2\left(\SE(3), \mu_{\Phi}\right)$.
   
   \item We verify all hypotheses of the abstract hypocoercivity method. As a consequence, we establish exponential convergence to equilibrium of the corresponding stochastic dynamics in $L^2\left(\SE(3), \mu_{\Phi}\right)$, assuming natural geometric and analytic conditions on the potential $\Phi$.~
\end{itemize}
%========================
% Preliminaries
%========================

\section{Preliminaries}

The Lie group $\SE(3)$ describes rigid motions in three dimensions and provides the natural geometric setting for coupled translational and rotation dynamics. An element of $\SE(3)$ consists of a position variable together with an orientation, and the corresponding invariant vector fields encode the interaction between translation and rotation. This structure plays a fundamental role in constructing the generator of the Kolmogorov backward operator, below, which naturally generalizes the $\SE(2)$ case in \cite{GHQ26}.

\begin{notation}\label{Notation_1}
The superscript ${}^{\top}$ denotes the transpose for matrices. In this article, vectors in $\mathbb{R}^3$ are considered as column matrices ($3 \times 1$ matrices). For $v, w \in \mathbb{R}^3$, the Euclidean inner product is denoted by $v \cdot w := v^{\top} w$, and the Euclidean norm by $|v|$.~

For matrices $A \in \mathbb{R}^{3 \times 3}$, $|A|$ denotes the Frobenius norm and $\operatorname{tr}(A)$ denotes the trace. For $v, w \in \mathbb{R}^3$, the outer product is defined as follows $v\otimes w:=vw^{\top}$, that is, $(v\otimes w)_{ij}=v_{i}w_j$.~
\end{notation}

%%%%%%%%%%%%%%%%%%%%%%%%
\subsection{Geometrical Preliminaries and Notations}\label{Geometrical-Preliminaries}

\begin{definition}[Euclidean Motion Lie group]
The Euclidean motion group $\SE(3)$ is the semidirect product of $\mathbb{R}^{3}$ with the special orthogonal group $\SO(3)$, denoted by $\SE(3)= \mathbb{R}^3 \rtimes \SO(3)$, and for elements $(\xi, R_1), (\eta, R_2 ) \in \SE(3)$ the group law $\circ$ is given by
\begin{equation*}
    (\xi, R_1) \circ (\eta, R_2 ) = (\xi+R_1\eta, R_1 R_2), \quad R_1,R_2 \in \SO(3), \quad \xi,\eta \in \mathbb{R}^3. 
\end{equation*}
\end{definition}

The geometric construction of the Kolmogorov operator relies on the right-invariant vector fields of $\SE(3)$.~Following \cite[Section 9.4]{Chi-Alex}, the rotation vector fields on $\SO(3)$ are denoted by $X_1, X_2, X_3$, while the translational directions, denoted by $\widetilde{X}_4, \widetilde{X}_5, \widetilde{X}_6$, are obtained through the action of $R^{\top}$ on the Euclidean gradient, where 

\begin{equation*}
\begin{aligned}
R&=\left(\begin{array}{ccc}
c_\alpha c_\gamma-c_\beta s_\alpha s_\gamma & -c_\alpha s_\gamma-c_\beta c_\gamma s_\alpha & s_\alpha s_\beta \\
c_\gamma s_\alpha+c_\alpha c_\beta s_\gamma & c_\alpha c_\beta c_\gamma-s_\alpha s_\gamma & -c_\alpha s_\beta \\
s_\beta s_\gamma & c_\gamma s_\beta & c_\beta
\end{array}\right).
%R^{\top} &=  \left(\begin{array}{ccc}
%c_\alpha c_\gamma-c_\beta s_\alpha s_\gamma & c_\gamma s_\alpha+c_\alpha c_\beta s_\gamma & s_\beta s_\gamma \\
%-c_\alpha s_\gamma-c_\beta c_\gamma s_\alpha &  c_\alpha c_\beta c_\gamma-s_\alpha s_\gamma & c_\gamma s_\beta \\
%s_\alpha s_\beta & -c_\alpha s_\beta & c_\beta
%\end{array}\right).
\end{aligned}
\end{equation*}

The matrix $R^{} \in \SO(3)$ is parametrized using the ZXZ Euler angles $\alpha, \beta, \gamma$, with $0 \leq \alpha, \gamma \leq 2\pi$ and $0 \leq \beta \leq \pi$.~The notation $c_j$ and $s_j$ refers to $\cos(j)$ and $\sin(j)$, respectively, for $j \in \{\alpha, \beta, \gamma\}$.~This convention is maintained throughout the article for lengthy computations. Otherwise, trigonometric terms are written explicitly. The corresponding right-invariant vector fields on $\SE(3)$ are as follows:

\begin{equation}\label{eq:vector_fields_SE(3)_compactified_version}
\widetilde{X}_i :=\left\{\begin{array}{cc}
X_i
& \text { for } i=1,2,3 \\
\left(R^{\top} \nabla_{\xi}\right)_{i-3} & \text { for } i=4,5,6
\end{array}\right.,
\end{equation}

where $\left(\nabla_{\xi}\right)_i= \frac{\partial}{\partial\xi_i}$, $\nabla_{\xi}=(\partial_{\xi_{1}}, \partial_{\xi_{2}}, \partial_{\xi_{3}})^{\top}$.~Therefore, the vector fields given in Equation \eqref{eq:vector_fields_SE(3)_compactified_version} are given in coordinates as follows
%\begin{center}
\begin{equation}\label{eq:vector_fields_on_SE(3)}
\begin{aligned}
\widetilde{X}_1 &=-\cot \beta \sin \gamma \frac{\partial}{\partial \gamma}+\frac{\sin \gamma}{\sin \beta} \frac{\partial}{\partial \alpha}+\cos \gamma \frac{\partial}{\partial \beta}, \\
\widetilde{X}_2 &=-\cot \beta \cos \gamma \frac{\partial}{\partial \gamma}+\frac{\cos \gamma}{\sin \beta} \frac{\partial}{\partial \alpha}-\sin \gamma \frac{\partial}{\partial \beta},\\
\widetilde{X}_3 &=\frac{\partial}{\partial \gamma},\\
\widetilde{X}_4 &= (c_\alpha c_\gamma-c_\beta s_\alpha s_\gamma) \nabla_{\xi_1} +  (c_\gamma s_\alpha+c_\alpha c_\beta s_\gamma) \nabla_{\xi_2} + s_\beta s_\gamma \nabla_{\xi_3},\\
\widetilde{X}_5 &= (-c_\alpha s_\gamma-c_\beta c_\gamma s_\alpha)\nabla_{\xi_1} + (c_\alpha c_\beta c_\gamma-s_\alpha s_\gamma) \nabla_{\xi_2} +  c_\gamma s_\beta  \nabla_{\xi_3},\\
\widetilde{X}_6 &= s_\alpha s_\beta \nabla_{\xi_1} -c_\alpha s_\beta \nabla_{\xi_2} + c_\beta \nabla_{\xi_3}.      
\end{aligned}
\end{equation}
%\end{center}

\begin{remark}
The expressions in Equation \eqref{eq:vector_fields_on_SE(3)} provide the coordinate representations of the right-invariant vector fields in ZXZ Euler angles, valid for $0<\beta<\pi$.~The apparent singularities at $\beta=0$ and $\beta=\pi$, due to the factors $1 / \sin \beta$ and $\cot \beta$, result from the Euler-angle parametrization, not from the vector fields themselves. The fields $\widetilde{X}_1, \widetilde{X}_2, \widetilde{X}_3$ are globally defined, smooth, right-invariant vector fields on $\operatorname{SO}(3)$.~The sets $\{\beta=0\}$ and $\{\beta=\pi\}$ have zero measure with respect to the Haar measure $\nu_3$ on $\SO(3)$, so the coordinate expressions are well defined $\nu_3$-almost everywhere for the integral computations below.    
\end{remark}

The commutator relations describing the Lie algebra structure of $\SE(3)$ are given by
\begin{equation}\label{eq:commutator-relations}
\begin{aligned}
\relax[\widetilde{X}_i, \widetilde{X}_j]  & =\varepsilon_{i j k} \widetilde{X}_k, \quad [ \widetilde{X}_i, \widetilde{X}_{j+3}]   =\varepsilon_{i j k} \widetilde{X}_{k+3}, \quad  i, j, k \in\{1,2,3\}, \\
%[\widetilde{X}_i^r, \widetilde{X}_{j+3}^r]  & =\varepsilon_{i j k} \widetilde{X}_{k+3}^r, & i, j, k \in\{1,2,3\},\\
&[\widetilde{X}_{i+3}, \widetilde{X}_{j+3}]=0, \quad i, j \in\{1,2,3\}.
\end{aligned}
\end{equation}

Let us define 
\begin{equation}\label{def:v_{i}_and_w_{i}}
w_{i} := v_{i} \cdot \nabla_{\xi}, \quad \text{where} \quad v_{i} := Re_{i}, \quad \text{for} \quad i=1,2,3.  
\end{equation}

According to Notation \ref{Notation_1}, the inner product should be interpreted as $w_{i} = v_{i} \cdot \nabla_{\xi} = v_{i}^{\top}\nabla_{\xi}$.~The vectors $v_{i}$ form an orthonormal frame in $\mathbb{R}^3$ defined by the rotation matrix $R \in \SO(3)$.~

\begin{remark}\label{rem:important-relations}
We observe that 

\begin{equation}
w_{i}= (R e_{i}) \cdot \nabla_{\xi} = \widetilde{X}_{i+3}, \quad i=1,2,3. \end{equation}
\end{remark}

\begin{observation}\label{observation_1}
For the vectors $v_{i}= Re_{i}$ defined above in Equation \eqref{def:v_{i}_and_w_{i}}, the following expressions hold:
\begin{equation}
v_{1} \times v_{2} =v_{3}, \quad v_{1} \times v_{3}= - v_{2}, \,\, \text{and} \,\, v_{2} \times v_{3}=v_{1}.  
\end{equation}

Here, $\times$ denotes the standard cross product in $\mathbb{R}^3$.~We use these expressions throughout the paper.
\end{observation}

The observation above gives us the following important relationship for the vector field 
\begin{equation}\label{X_6}
\widetilde{X}_6= v_{3} \cdot \nabla_{\xi}= \left( v_{1} \times v_{2} \right) \cdot \nabla_{\xi}, 
\end{equation}    

which is useful for constructing the antisymmetric operator in Proposition \ref{Anti_Construction}, below.

\begin{lemma}\label{lemma_2_important} Let $v_{i}= R e_{i}$ for $i=1,2,3$.~ Then
\begin{equation*}
\widetilde{X}_1\left(v_{2}\right)=v_{3}, \quad \widetilde{X}_1\left(v_{3}\right)=-v_{2}, \quad \widetilde{X}_2\left(v_{1}\right)=-v_{3}, \quad \widetilde{X}_2\left(v_{3}\right)=v_{1}, \quad \widetilde{X}_3\left(v_{3}\right)=0 .    
\end{equation*}

These identities follow by direct differentiation of the explicit expressions of the vectors $v_{i}$ in Euler angles.
\end{lemma}

%------------------------------------------------------
%\newpage
\subsection{The abstract hypocoercivity method}\label{The_Hilbert_Method_Space}

In this section, we briefly recall the concepts of the abstract hypocoercivity method from \cite{GROT-STIL-2014}.

\noindent
\subsubsection*{Hypocoercivity Data (D)}\label{Hypocoercivity_Data_Overview}

The abstract hypocoercivity method is based on the following data conditions:

\begin{description}[leftmargin=1.3cm, labelindent=0cm, style=nextline]
\item[\textbf{(D1).}]  \textbf{Hilbert Space}:
The Hilbert space $\Hi = L^2(E, \mu)$, with $(E, \mathcal{F}, \mu)$ a probability space.

\item[\textbf{(D2).}] \textbf{Semigroup and Generator}: The operator $(\L, D(\L)) $ generates a strongly continuous semigroup  $(T_t)_{t \ge 0}$ on $\Hi$.~

\item[\textbf{(D3).}] \textbf{Core Property}:
$ \D \subset D(\L)$ is dense in $\Hi$ and a core for $(\L, D(\L))$.~

\item[\textbf{(D4).}] \textbf{Decomposition}:
There exist operators $\S$ and $ \A$ such that 
\begin{equation*}
\D \subset D(\S) \cap D(\A), 
\quad 
\L|_\D = \S - \A,
\end{equation*}
with $ (\S, D(\S))$ symmetric, $ (\A, D(\A))$ antisymmetric and closed on $\Hi$.~

\item[\textbf{(D5).}]\textbf{Projection}:
There exists an orthogonal projection $ \P : \Hi \to \Hi $ satisfying
\begin{equation*}
\P (\Hi) \subset D(\S), \quad \S \P = 0,
\end{equation*}
and
\begin{equation*}
\P (\D) \subset D(\A), \quad \A \P (\D) \subset D(\A).
\end{equation*}

Define
\begin{equation*}
\Pi_{\S} f := \Pi f + (f,1)_{\Hi}, \quad f\in \Hi.
\end{equation*}
%then $\Pi_S$ is the orthogonal projection onto $\ker(S)$.

\item[\textbf{(D6).}]\label{D6} \textbf{Invariant Measure}:
The measure $ \mu$ is invariant, i.e.,
\begin{equation*}
\int_E \L f \, d\mu = 0, \quad 
\text{for all} \quad  f \in \D.
\end{equation*}

\item[\textbf{(D7).}] \textbf{Conservativity}:
$1 \in D(\L)$ and $\L 1 = 0$.~
\end{description}

Henceforth, we denote the core by $\D$, and the domain of the operator $\L$ by $D(\L)$.~
%---------------------------------
%\textbf{The hypocoercivity conditions}\\
\subsubsection*{Hypocoercivity Assumptions (H)}\label{Hypocoervicity_Assumptions_SE(n)}

Under the above data conditions, hypocoercivity is ensured provided that the following assumptions hold:

\begin{itemize}

\item[\textbf{(H1).}]   \textbf{Algebraic Relation}:
\begin{equation}
\P \A \P|_\D = 0.   
\end{equation}

\item[\textbf{(H2).}]  \textbf{Microscopic Coercivity}:
There exists $\Lambda_m > 0$ such that
\begin{equation}
-(\S f, f)_{\Hi} \ge \Lambda_m \| (\I - \P_{\S}) f \|^2_{\Hi}, 
\quad 
\text{for all} \quad  f \in \D.
\end{equation}

\item[\textbf{(H3).}]  \textbf{Macroscopic Coercivity}:
There exists $\Lambda_M >0$ such that
\begin{equation}\label{macro_5}
\| \A \P f \|^2_{\Hi} \ge \Lambda_M \| \P f \|^2_{\Hi},
\quad \text{for all} \quad f \in D\big((\A \P)^*(\A \P)\big).
\end{equation}

\item[\textbf{(H4).}]  \textbf{Auxiliary Operator Bounds}:
Let the operator $B$ be defined as
\begin{equation*}
\B:=\left(\I+(\A \P)^* \A \P\right)^{-1}(\A \P)^* \quad \text { on } \quad D\left((\A \P)^*\right) .
\end{equation*}

The operators $(\B \S, \D)$ and $(\B \A(\I - \P), \D)$ are bounded. Moreover, there exist constants $c_1 < \infty$ and $c_2 < \infty$ such that 
\begin{equation}\label{aucxill}
\begin{aligned}
\|\B\S f\|_{\Hi} &\leq c_1 \|(\I - \Pi_{\#}) f\|_{\Hi},\\
\|\B\A(I - \P) f\|_{\Hi} &\leq c_2 \|(\I - \Pi_{\#}) f\|_{\Hi}, \quad \text{for} \quad f \in \D, 
\end{aligned}
\end{equation}
where $\Pi_{\#}$ is either $\P_{\S}$ or $\P$.~
\end{itemize}

To verify the above hypocoercivity assumptions in concrete settings, we recall some auxiliary criteria from \cite{GROT-STIL-2014}.

\begin{corollary}\label{cor:cor2.15P}
Let the inequality in \eqref{macro_5} be satisfied for all $f \in \D$, and assume that the operator $(G, \D):=\left(\P \A^2 \P, \D\right)$ is essentially m-dissipative (or essentially selfadjoint, equivalently) in $\Hi$, i.e., $(\I-G)(\D)$ is dense in $\Hi$.~Then Condition \textnormal{\textbf{(H3)}} is fulfilled.
\end{corollary}

\begin{lemma}\label{auxiliary_lemma_1}
Suppose that Condition \textnormal{\textbf{(H1)}} is satisfied. Assume that $\S(\D) \subset D(\A)$ and assume that there exists $c_3 \in \mathbb{R}$ such that
\begin{equation*}
\P \A \S=c_3 \P \A \quad \text {on} \quad \D .
\end{equation*}

Then $(\B \S, \D)$ is bounded and the first inequality in \textnormal{\textbf{(H4)}} holds with $c_1=\frac{1}{2}\left|c_3\right|$.~
\end{lemma}
 
\begin{proposition}\label{auxiliary-prop2.17}
Let $f \in \D$.~For each $g=(\I-G) f$  we have
\begin{equation}
g \in D\left((\B \A)^*\right) \quad \text { with } \quad(\B \A)^* g=-\A^2 \P f.
\end{equation}

Assume that $(G, \D)$  is essentially m-dissipative (or essentially selfadjoint, respectively) and assume that there exists $c_4<\infty$ such that
\begin{equation}
\left\|(\B \A)^* g\right\|_{\Hi}=\left\|\A^2 \P f\right\|_{\Hi} \leq c_4\|g\|_{\Hi} \quad \text {for all} \quad g=(\I-G) f, \, f \in \D .
\end{equation}

Then the second inequality in Condition \textnormal{\textbf{(H4)}} is satisfied with $c_2=c_4$.~
\end{proposition}

%------------
%\nonindent
%\textbf{The hypocoercivity theorem}

\begin{theorem}\label{Hypocoercive_them_for_SE(3)}
Assume that the hypocoercivity data \textnormal{\textbf{(D)}} and the hypocoercivity Assumptions \textnormal{\textbf{(H1)}}--\textnormal{\textbf{(H4)}} hold. Then there exist constants $\kappa_1, \kappa_2 \in (0, \infty)$, depending only on $\Lambda_m, \Lambda_M, c_1$ and $c_2$ such that for all $f \in \Hi$, the associated semigroup $(T_t)_{t\ge0}$ satisfies
\begin{equation*}\label{eq:exponential_convergence_to_equilibrium}
    \| T_t f - (f,1)_{\Hi}\|_{\Hi} \leq \kappa_1 e^{-\kappa_2 t} \|f - (f,1)_{\Hi}\|_{\Hi} \quad \text{for all } \quad t \geq 0. 
\end{equation*}
\end{theorem}

The constants $\kappa_1$ and $\kappa_2$ depend on the micro and macro coercivity constants $\Lambda_m$ and $\Lambda_M$, as well as on $c_1$ and $c_2$ from Assumption \textnormal{\textbf{(H4)}}. Remark \ref{rem:final}, below, provides a more detailed description of these constants.

%\newpage
\section{Hypocoercivity on the  Euclidean motion group \texorpdfstring{$\SE(3)$}{SE(3)}}\label{Hypo_SE(3)}

This section contains the core of the present article, and its goal is to give an explicit construction of the generator of a hypocoercive Langevin-type dynamics on the Lie group $\SE(3)$.~Throughout this section, we use the explicit right-invariant vector fields on $\SE(3)$ introduced in the preliminaries, see Equation \eqref{eq:vector_fields_SE(3)_compactified_version}.
%-----------------------------------------------------------------------------------------------------
\subsection{The symmetric part of the generator}

We first specify the invariant probability measure, the corresponding Hilbert space of square-integrable and core of nice functions, and the symmetric part of the generator.

\begin{definition}\label{def:measure_SE(3)}
Let $\Phi\in C^\infty(\mathbb{R}^3)$ be a real valued potential,
such that there exists a finite normalization constant $Z(\Phi):=\int_{\mathbb{R}^3}e^{-\Phi(\xi)} d\xi <\infty$.~We define the probability measure $\mu_{\Phi}$ on $\SE(3)$ by
\begin{equation}\label{eq:measure_SE(3)}
    \mu_\Phi:=\gamma_\Phi\otimes\nu_3,
\end{equation}
where $\gamma_\Phi$ is the probability measure on $\mathbb{R}^3$
given by
\begin{equation}
    \gamma_\Phi(B_1)
    :=
    Z(\Phi)^{-1}
    \int_{B_1}e^{-\Phi(\xi)}\,d\xi,
 \quad B_1\in\Bor(\mathbb{R}^3), \end{equation}
and $\nu_3$ is the normalized Haar measure on $\SO(3)$, given in
ZXZ Euler angles by
\begin{equation}
    \nu_3(B_2)
    :=
    \frac{1}{8\pi^2}
    \int_{B_2}
    \sin(\beta)\,d\alpha\,d\beta\,d\gamma,
    \quad
    B_2\in\Bor(\SO(3)).
\end{equation}
\end{definition}

\begin{remark}
The choice of this probability measure is quite natural, given that our space is isomorphic to a non-compact component $\mathbb{R}^3$, and a compact component $\SO(3)$. We want to control the non-compact part with a factor $e^{-\Phi}$, for some potential $\Phi$ as defined above. 

The normalized Haar measure $\nu_{3}$ on $\SO(3)$ in terms of the Euler angles $\alpha, \beta, \gamma$ is given explicitly in \cite[Chapter 5, Section 5.5.1]{Chi-Alex}.
\end{remark}

We consider the Hilbert space $\Hi= L^2(\SE(3),\mu_{\Phi})$, with core $\D := C_{c}^{\infty}(\SE(3))$, the smooth compactly supported functions on $\SE(3)$.~For $\sigma \in (0, \infty)$, we define the symmetric operator $(\S,\D)$ in $\Hi$ as a Laplacian on $\SO(3)$, i.e.,

\begin{equation}\label{eq:symmetric_operator_SE(3)}
\begin{aligned}
\S := -\frac{\sigma^{2}}{2} \sum_{i=1}^3 \widetilde{X}_{i}^2 = \frac{\sigma^{2}}{2}\Delta_{\SO(3)}.  
\end{aligned}
\end{equation}

Therefore, the symmetric operator $\S$ is expressed just in terms of the rotation vector fields $\{\widetilde{X}_i\}_{i=1,2,3}$  on $\SO(3)$.

\subsection{Construction of the antisymmetric part of the generator}\label{construction:antisymmetric-operator}

In contrast, the construction of the antisymmetric operator $\A$ is more delicate. In view of the hypocoercivity data introduced in Section~\ref{Hypocoercivity_Data_Overview}, it should be chosen to satisfy two structural requirements:

\begin{itemize} 

\item Antisymmetry with respect to the probability measure $\mu_\Phi$, and consequently compatibility with the invariance property required in \textbf{(D6)}. 

\item Generating the missing vector fields corresponding to the translation directions through commutators with the vector fields of $\S$, ensuring that the operator $\L$ satisfies Hörmander's condition. It's implications we need to show essential m-dissipativity of the operator $(\L, \D)$. 
\end{itemize}

Recall from Remark \ref{rem:important-relations} that the vector field $\widetilde{X}_6$ can also be written as $w_3 = v_{3} \cdot \nabla_{\xi}$, where $v_{3} = Re_3^{}$.~Additionally, consider the commutator relations and the requirement for $\A$ to be an antisymmetric operator on $\D$.~However, with respect to the weighted measure as in Definition \ref{def:measure_SE(3)}, the vector field $\widetilde{X}_6$ is not antisymmetric with respect to the measure $\mu_\Phi$.~To address this, we consider an operator of the form
\begin{equation}\label{A_0}
    \A_0 := - \widetilde{X}_6 + T, 
\end{equation}
where $T$ is a first-order operator on $\D$.~To correct this, we add a rotation term that acts only on the rotation directions. The natural geometric quantity is $\nabla_{\xi}\Phi \times v_{3}$, which is orthogonal to $v_{3}$.~Therefore, the correction is confined to the rotation directions generated by $\widetilde{X}_1$ and $\widetilde{X}_2$.~The construction of the antisymmetric operator uses the geometric identities established in Section \ref{Geometrical-Preliminaries}.

%%%%%%%%%%%%%%%%%%%%%%%%%%%%%%%%%%%%%%%%%%%%%%%%%%%%%%%%%%%%%%%%%%%%%%%%%%%%%%%%%%%%%%%%%

\begin{proposition}\label{decomposition-of-cross-product}
Let $v_{i}=Re_{i}$, $i=1,2,3$, where $R\in \SO(3)$.~Then the vector
$\nabla_{\xi}\Phi(\xi)\times v_{3}$, belongs to the plane spanned by $v_{1}$ and $v_{2}$, and admits the decomposition
\begin{equation*}
\nabla_{\xi}\Phi \times v_{3} = \tau_{1} \,v_{1}
+ \tau_{2} \,v_{2}, 
%(\nabla_{\xi}\Phi(\xi)\cdot v_{2})\,v_{1}-(\nabla_{\xi}\Phi(\xi)\cdot v_{1})\,v_{2}.
\end{equation*}
where $\tau_{1}:=\nabla_{\xi} \Phi \cdot v_{2}$, and $\tau_{2}:=-\nabla_{\xi}\Phi \cdot v_{1}$.~
\end{proposition}

%\unsure{$\tau_{1}$, where ... }

\begin{proof}
Since $\{v_{1},v_{2},v_{3}\}$ is an orthonormal basis of $\mathbb R^3$, every vector admits the decomposition
%\unsure{The transpose.}
\begin{equation*}
\nabla_{\xi}\Phi \times v_{3}
=
\sum_{i=1}^3
\bigl(
(\nabla_{\xi}\Phi \times v_{3})\cdot v_{i}
\bigr)v_{i} .
\end{equation*}

Moreover, $(\nabla_{\xi}\Phi \times v_{3})\cdot v_{3}=0$, since the cross product is orthogonal to both factors. Therefore,
\begin{equation}\label{eq:nabla_tau}
\nabla_{\xi}\Phi \times v_{3} = \tau_{1}v_{1}+\tau_{2}v_{2},
\end{equation}

where the terms $\tau_{1}$ and $\tau_{2}$ are given by $\tau_{1}
=(\nabla_{\xi}\Phi \times v_{3})\cdot v_{1}$ and $\tau_{2}=(\nabla_{\xi}\Phi \times v_{3})\cdot v_{2}$.~Using the vector identity $(a\times b)\cdot c =a\cdot(b\times c)$,
together with $v_{3}\times v_{1}=v_{2}, v_{3}\times v_{2}=-v_{1}$ in Observation \ref{observation_1}, we obtain
\begin{equation*}
\begin{aligned}
\tau_{1} &=
(\nabla_{\xi}\Phi \times v_{3})\cdot v_{1}
=\nabla_{\xi}\Phi \cdot(v_{3}\times v_{1})
=\nabla_{\xi}\Phi \cdot v_{2}, \\    
\tau_{2} &=(\nabla_{\xi}\Phi \times v_{3})\cdot v_{2}
=\nabla_{\xi}\Phi \cdot(v_{3}\times v_{2})
=-\nabla_{\xi}\Phi \cdot v_{1}.
\end{aligned}
\end{equation*}

Hence Equation \eqref{eq:nabla_tau} can be written as
\begin{equation*}
\begin{aligned}
\nabla_{\xi}\Phi \times v_{3}
=
(\nabla_{\xi}\Phi \cdot v_{2})\,v_{1}
-
(\nabla_{\xi}\Phi \cdot v_{1})\,v_{2}.
\end{aligned}
\qedhere
\end{equation*} \end{proof}

Therefore, the candidate for the antisymmetric operator on $\D$ is given by
\begin{equation}\label{anti:construction}
    \begin{aligned}
    \A_{\alpha} &:= - (R e_{3}) \cdot \nabla_{\xi}  + \alpha \left( \tau_{1} \widetilde{X}_{1} + \tau_{2} \widetilde{X}_{2} \right)\\  
    %&= v_{3} \cdot \nabla_{\xi}  + \alpha  \left( (\nabla_{\xi} \Phi(\xi) \cdot v_{2}) \widetilde{X}_{1}^r - (\nabla_{\xi} \Phi(\xi) \cdot v_{1}) \widetilde{X}_{2}^r \right). \\
    &= - \widetilde{X}_{6}   + \alpha  \left( \nabla_{\xi} \Phi \cdot v_{2} \,\, \widetilde{X}_{1} - \nabla_{\xi} \Phi \cdot v_{1} \,\, \widetilde{X}_{2} \right). 
    \end{aligned}
\end{equation}
%\improvement{This is one of the most important results in the paper, check very well the convention and also, after this proposition use the conventions, or remind them that they are in chapter 2.}
The goal of the following proposition is to find the correct value for $\alpha$ above. 
%%%%%%%%%%%%%%%%%%%%%%%%%%%%%%%%%%%%%%%%%%%%%%%%%%%%%%%%%%%%%%%%%%%%%%%%%%%%%%%%%%%%%%%%%

\begin{proposition}\label{Anti_Construction}

The operator $\left(\A_\alpha, \D\right)$ in Equation \eqref{anti:construction} is antisymmetric on $\Hi$ if and only if $\alpha=-\frac{1}{2}$.~Consequently, we define $\A:= \A_{- \frac{1}{2}}$.~ Thus, using  $\widetilde{X}_6=v_{3}\cdot\nabla_{\xi}$,

\begin{equation*}
    \begin{aligned}
    \A &= - \widetilde{X}_6  - \frac{1}{2} \left( \nabla_{\xi} \Phi \cdot v_{2} \,\, \widetilde{X}_{1} - \nabla_{\xi} \Phi \cdot v_{1} \,\, \widetilde{X}_{2} \right).  
    \end{aligned}
\end{equation*}

\end{proposition}

\begin{proof}

%\textbf{How to find $a$ such that the resulting operator is antisymmetric?}

Let us consider the expression of the operator $\A_\alpha$ above in \eqref{anti:construction}, under the requirement that the operator is antisymmetric with respect to the weighted measure $\mu_\Phi$.~ Specifically, for $f_1, f_2 \in \D$, we require $(\A_\alpha f_1, f_2)_{\Hi} = - (f_1, \A_\alpha f_2)_{\Hi}$.~ Therefore,
%\begin{equation}
 %   \begin{aligned}
  %  &(\A_\alpha f,g)_{\Hi} = Z(\Phi)^{-1} \int_{\mathbb{R}^3 \rtimes \SO(3)} \A_\alpha f \, g  \,\, d \mu_{\Phi}\\
   % &=    \underbrace{- Z(\Phi)^{-1}  \int_{\mathbb{R}^3 \rtimes \SO(3)}(Re_3) \cdot \nabla_{\xi} f \,\,  g  \quad  \mu_{\Phi}}_\textbf{$=:$ TI$(f_{1}, f_{2})$ } \\
    %&+ \underbrace{ \alpha \, Z(\Phi)^{-1}  \int_{\mathbb{R}^3 \rtimes \SO(3)} \left( \nabla_{\xi} \Phi  \times R e_{3} \right) \cdot \nabla_{\SO(3)}  f \,\, g  \quad  \mu_{\Phi}}_\textbf{$=:$ TII$(f_{1}, f_{2})$ }. 
    %\end{aligned}
%\end{equation}
\begin{equation}
    \begin{aligned}
    &(\A_\alpha f_1,f_2)_{\Hi} =  \int_{\SE(3)} \A_\alpha f_1 \, f_2  \,\, d \mu_{\Phi}\\
    &=    \underbrace{-   \int_{\SE(3)}(Re_3) \cdot \nabla_{\xi} f_1 \,\,  f_2  \,\, d \mu_{\Phi}}_\text{$=:$ TI$(f_{1}, f_{2})$ } + \underbrace{ \alpha \,   \int_{\SE(3)} \left( \nabla_{\xi} \Phi  \times R e_{3} \right) \cdot \nabla_{\SO(3)}  f_1 \,\, f_2  \,\,  d \mu_{\Phi}}_\text{$=:$ TII$(f_{1}, f_{2})$ }. 
    \end{aligned}
\end{equation}

Let us do integration by parts on \text{TI$(f_{1}, f_{2})$ } in the expression above
\begin{equation}
    \begin{aligned}
        &\text{TI$(f_{1}, f_{2})$ }= - \int_{\SE(3)} \left( Re_3 \cdot \nabla_{\xi} f_1  \right) \, f_2  \,\,  d \mu_{\Phi}\\
        %&=- Z(\Phi)^{-1} \int_{\mathbb{R}^3 \rtimes \SO(3)} \left( (Re_3)^{\top} \cdot \nabla_{\xi} f \right) \, g  \quad e^{-\Phi(\xi)} d \xi \, \nu_{3}\\
        &=  Z(\Phi)^{-1} \int_{\SO(3)} \int_{\mathbb{R}^3}  f_1 \,\, \left( Re_3 \cdot \nabla_{\xi} (f_2  e^{-\Phi })   \right)  \, d \xi \,\, d \nu_{3}\\ 
         &=   \underbrace{   Z(\Phi)^{-1} \int_{\SO(3)} \int_{\mathbb{R}^3}  f_1  (Re_3) \cdot \nabla_{\xi} f_2 \,\,  e^{-\Phi} d \xi \,\, d\nu_{3} }_\text{=$- \text{TI$(f_{2}, f_{1})$ }$}  \\
         &  - \underbrace{   Z(\Phi)^{-1} \int_{\SO(3)} \int_{\mathbb{R}^3}  f_1  (v_{3} \cdot \nabla_{\xi} \Phi )  f_2 \,\,  e^{-\Phi} d \xi \, d\nu_{3}}_\text{$=:$ETI}.
    \end{aligned}
\end{equation}

%\unsure{Too many we's.}

This operator must be antisymmetric with respect to the measure \eqref{eq:measure_SE(3)}. We now focus on the second term.
\begin{align*}\label{Term(II)}
   % \begin{aligned}
        &\text{TII$(f_{1}, f_{2})$ }=  \alpha \,   \int_{\SE(3)} \left( \nabla_{\xi} \Phi \cdot v_{2} \,\ \widetilde{X}_{1} f_1 - \nabla_{\xi} \Phi \cdot v_{1} \,\, \widetilde{X}_{2} f_1 \right)\,f_2  \,\,  d\mu_{\Phi}\\
        %&=  - \alpha \, Z(\Phi)^{-1}  \int_{\mathbb{R}^3 \rtimes \SO(3)}   f    \left(  \widetilde{X}_{1}^r  ((\nabla_{\xi} \Phi(\xi) \cdot v_{2}) g  ) -  \widetilde{X}_{2}^r ((\nabla_{\xi} \Phi(\xi) \cdot v_{1})g ) \right)   \quad \mu_{\Phi}\\
        &= -  \alpha \, \int_{\SE(3)}   f_1    \left(   \nabla_{\xi} \Phi \cdot \widetilde{X}_{1} (v_{2}) \,\, f_2    +  \nabla_{\xi} \Phi \cdot v_{2} \,\, \widetilde{X}_{1} f_2    \right. \\
        &  \hspace{2cm} \left. -  \nabla_{\xi} \Phi \cdot \widetilde{X}_{2} (v_{1}) \, f_2 +   \nabla_{\xi} \Phi \cdot v_{1} \,\, \widetilde{X}_{2} f_2 \right)   \, d\mu_{\Phi}\\
        &=    \left.  \underbrace{ - \alpha \,   \int_{\SE(3)}   f_1  \left(\nabla_{\xi} \Phi \cdot \widetilde{X}_{1} (v_{2})  -  \nabla_{\xi} \Phi \cdot \widetilde{X}_{2} (v_{1}) \right)f_2    \, d\mu_{\Phi} }_\text{$=:$ETII}    \right. \\
        & \hspace{2cm} \left. + \underbrace{ \alpha \,   \int_{\SE(3)}   f_1 \left( \nabla_{\xi} \Phi \cdot v_{2} \,\, \widetilde{X}_{1} f_2   -  \nabla_{\xi} \Phi \cdot v_{1} \,\, \widetilde{X}_{2} f_2 \right)  \, d\mu_{\Phi} }_\text{$=-\text{TII$(f_{2}, f_{1})$ }$} \right. .
    %\end{aligned}
\end{align*}

%-----------
From Lemma \ref{lemma_2_important}, and the right-invariant rotation vector fields $\widetilde{X}_1,\widetilde{X}_2,\widetilde{X}_3$ on $\SO(3)$ we recall that $\widetilde{X}_1(v_{2})=v_{3},\quad \widetilde{X}_2(v_{1})=-v_{3}$.~Therefore, the extras terms in \text{TI$(f_{1}, f_{2})$ } and \text{TII$(f_{1}, f_{2})$ } is given as follows 
\begin{equation}
    \begin{aligned}
       \text{ETI} + \text{ETII} &= \int_{\SE(3)}  f_1 \,\, \Big( - v_{3} \cdot \nabla_{\xi} \Phi   - \alpha \Big(    \nabla_{\xi} \Phi \cdot v_{3}    -  \nabla_{\xi} \Phi \cdot  (-v_{3}) \Big) \Big) f_2 \, d\mu_\Phi  \\
       &=  \int_{\SE(3)}  f_1 \,\, \Big( - v_{3} \cdot \nabla_{\xi} \Phi  - \alpha \left( 2  \nabla_{\xi} \Phi \cdot v_{3} \right) \Big) f_2 \, d\mu_\Phi  \\
       %&=  \int_{\SE(3)}  f \,\, \Big( - v_{3} \cdot \nabla_{\xi} \Phi  - 2 \alpha \left( \nabla_{\xi} \Phi \cdot v_{3}  \right)  \Big) g \,  d\mu_\Phi  \\
       &= - (1+2 \alpha)\int f_1 f_2\, \left(\nabla_{\xi}\Phi\cdot v_{3} \right)\, d\mu_\Phi.
    \end{aligned}
\end{equation}

Hence, for $\alpha = - \frac{1}{2}$, the extra term vanishes, and the statement follows. 
\end{proof}

\begin{remark}
The correction term has a torque-like interpretation. The vector
$\nabla_{\xi}\Phi\times v_{3}$ is perpendicular to the transport direction $v_{3}$ and lies in the plane spanned by $v_{1}$ and $v_{2}$. Hence the correction acts only through the rotation vector fields already present in the symmetric part. In this sense, $\A$ is the minimal correction of the transport field $v_{3}\cdot\nabla_{\xi}$ that restores antisymmetry with respect to $\mu_\Phi$.~% The identities $\tilde X_1(v_{2})=v_{3}$ and $\tilde X_2 (v_{1})=-v_{3}$ are verified in Section \ref{Geometrical-Preliminaries} Lemma \ref{lemma_2_important}.
\end{remark}

%\newpage
\subsection{Verification of hypocoercivity data (D)}

Finally, we can define our Kolmogorov backward operator $(\L,\D)$ by
\begin{equation}\label{eq:Operator_L}
\L:= S - A = \frac{\sigma^2}{2} \Delta_{\SO(3)} + \widetilde{X}_6  + \frac{1}{2}  \left( \nabla_{\xi} \Phi \cdot v_{2} \,\, \widetilde{X}_{1} - \nabla_{\xi} \Phi \cdot v_{1} \,\,  \widetilde{X}_{2} \right)
\end{equation}
on $\D= C_c^\infty(\SE(3))$.

The following lemma summarises basic properties of the operator $(\L, \D )$.~

\begin{lemma}\label{lemma:lemma2.25,SE(3)}
The operators $(\S, \D)$ and $(\A, \D)$, and the measure $\mu_\Phi$, satisfy:

\begin{enumerate}
\item[(i)] $(\S, \D)$ is symmetric and non-positive definite in $\Hi$.~

\item[(ii)] $(\A, \D)$ is antisymmetric in $\Hi$.~ 

\item[(iii)] $\mu_{\Phi}$ is invariant for $(\L, \D)$ in the sense that $(\L f, 1)_{\Hi} = \int_{\SE(3)}  \L f d \mu_{\Phi} =0$ for all $f \in \D$.~ 

\end{enumerate}
\end{lemma}

\begin{proof}
(i) Let $f_1,f_2 \in C_c^{\infty}(\SE(3))$, using Green's formula and Fubini's theorem, we obtain the following.
\begin{equation*}
\begin{aligned}
(\S f_1, f_2)_{\Hi} &= \frac{\sigma^2}{2}(\Delta_{\SO(3)} f_1, f_2)_{\Hi} = - \frac{\sigma^2}{2}(\nabla_{\SO(3)} f_1, \nabla_{\SO(3)} f_2)_{\Hi}\\
&= \frac{\sigma^2}{2} (f_1, \Delta_{\SO(3)}  f_2)_{\Hi} = ( f_1, \S f_2)_{\Hi}
    \end{aligned}
\end{equation*}
From this we also have that
\begin{equation*}
(\mathcal{S} f, f)_{\Hi}=-\frac{\sigma^2}{2}\left\|\nabla_{\SO(3)} f\right\|_{\Hi}^2 \leq 0. 
\end{equation*}
Hence $(\S, \mathcal{D})$ is a symmetric operator and non-positive.\\

(ii) This is proven by the construction of the antisymmetric operator in Section \ref{construction:antisymmetric-operator}.\\

(iii) Let $f \in \mathcal{D}$.~Let us note that for the symmetric operator we have that 
\begin{equation}\label{TS}
\begin{aligned}
(\S f, 1)_{\Hi} & =\frac{\sigma^2}{2} \int_{\mathbb{R}^3} \int_{\SO(3)} \Delta_{\mathrm{SO}(3)} f \, 1\,  d \nu_3 d \gamma_{\Phi} \\
& =-\frac{\sigma^2}{2} \int_{\mathbb{R}^3} \int_{SO(3)} \nabla_{\mathrm{SO}(3)} f \cdot \nabla_{\SO(3)} 1  \, d \nu_3 d \gamma_{\Phi} \\
& =0.
\end{aligned}
\end{equation}

For the antisymmetric operator, we obtain:
\begin{equation}
\begin{aligned}
&(\A f, 1)_{\Hi} = \underbrace{-\int_{\SE(3)} \widetilde{X}_6 f d \mu_{\Phi} }_\text{$=:$TI}\\
 & \underbrace{-\frac{1}{2} \int_{\SE(3)}\left(\nabla_{\xi} \Phi \cdot v_2\right) \widetilde{X}_1 f d \mu_{\Phi} }_\text{$=:$TII}  + \underbrace{\frac{1}{2} \int_{\SE(3)} \left(\nabla_{\xi} \Phi \cdot v_1\right) \widetilde{X}_2 f d \mu_{\Phi}}_\text{$=:$TIII}.
\end{aligned}
\end{equation}

\textbf{Term TI.} Integrating by parts in $\xi$ yields
\begin{equation}
\begin{aligned}
\text{TI} &= -\int_{\mathbb{R}^3} v_3 \cdot \nabla_{\xi} f d \gamma_{\Phi}  =-Z(\Phi)^{-1}  \int_{\mathbb{R}^3} v_3 \cdot \nabla_{\xi} f e^{-\Phi} d \xi \\
& =Z(\Phi)^{-1} \int_{\mathbb{R}^3} f v_3 \cdot \nabla_{\xi}\left(e^{-\Phi}\right) d \xi 
=\int_{\mathbb{R}^3} f\left(\nabla_{\xi} \Phi \cdot v_3\right) d \gamma_{\Phi}.
\end{aligned}
\end{equation}

Thus,
\begin{equation}\label{TI}
\text{TI} %= -\int_{ \SE(3)} \widetilde{X}_6 f d \mu_{\Phi}
=- \int_{\SE(3)} f\left(\nabla_{\xi} \Phi \cdot v_3\right) d \mu_{\Phi}.
\end{equation}

\textbf{Term TII.} Since $\widetilde{X}_1$ is a rotation invariant vector field, integrating by parts with respect to the normalized Haar measure $\nu_3$ yields
\begin{equation}
\int_{\SO(3)} h \widetilde{X}_1 f d \nu_3=-\int_{\SO(3)} f \widetilde{X}_1 h d \nu_3 \quad \text{for } \, f \in \D , h \in C^{\infty}(\SE(3)).
\end{equation}

Therefore,
\begin{equation}
\begin{aligned}
\text{TII} = -\frac{1}{2} \int_{\SE(3)}\left(\nabla_{\xi} \Phi \cdot v_2\right) \widetilde{X}_1 f d \mu_{\Phi} =\frac{1}{2} \int_{\SE(3)} f \widetilde{X}_1\left(\nabla_{\xi} \Phi \cdot v_2\right) d \mu_{\Phi}.
\end{aligned}
\end{equation}

Because the potential $\Phi$ depends only on $\xi$, $\widetilde{X}_1 \nabla_{\xi} \Phi=0$, so
\begin{equation}
\widetilde{X}_1\left(\nabla_{\xi} \Phi \cdot v_2\right)=\nabla_{\xi} \Phi \cdot \widetilde{X}_1 v_2 .
\end{equation}

Using the expression $\widetilde{X}_1 (v_2)=v_3$ in Lemma  \ref{lemma_2_important}, we obtain
\begin{equation}\label{TII}
\text{TII} %= -\frac{1}{2} \int_{\SE(3)}\left(\nabla_{\xi} \Phi \cdot v_2\right) \widetilde{X}_1 f d \mu_{\Phi}
=\frac{1}{2} \int_{\SE(3)} f\left(\nabla_{\xi} \Phi \cdot v_3\right) d \mu_{\Phi}.
\end{equation}

\textbf{Term TIII.} Similarly, for the last term,
\begin{equation}
\begin{aligned}
\text{TIII} = \frac{1}{2} \int_{\SE(3)}\left(\nabla_{\xi} \Phi \cdot v_1\right) \widetilde{X}_2 f d \mu_{\Phi} &=-\frac{1}{2} \int_{\SE(3)} f \widetilde{X}_2\left(\nabla_{\xi} \Phi \cdot v_1\right) d \mu_{\Phi} \\
&=-\frac{1}{2} \int_{\SE(3)} f\left(\nabla_{\xi} \Phi \cdot \widetilde{X}_2 v_1\right) d \mu_{\Phi}.
\end{aligned}
\end{equation}

Given that $\widetilde{X}_2 (v_1)=-v_3$, we conclude that
\begin{equation}\label{TIII}
\text{TIII} %= \frac{1}{2} \int_{\SE(3)}\left(\nabla_{\xi} \Phi \cdot v_1\right) \widetilde{X}_2 f d \mu_{\Phi}
=\frac{1}{2} \int_{\SE(3)} f\left(\nabla_{\xi} \Phi \cdot v_3\right) d \mu_{\Phi}.
\end{equation}

Combining Equations \eqref{TI}, \eqref{TII}, and \eqref{TIII}, we obtain
\begin{align*}
(\A f,1)_{\Hi} &= \text{TI} + \text{TII} + \text{TIII}
%&=-\int_{\SE(3)} f(\nabla_{\xi}\Phi\cdot v_3)\,d\mu_\Phi
%+\frac12\int_{\SE(3)} f(\nabla_{\xi}\Phi\cdot v_3)\,d\mu_\Phi\\
%&\quad
%+\frac12\int_{\SE(3)}
%f(\nabla_{\xi}\Phi\cdot v_3)\,d\mu_\Phi\\
=\left(-1+\frac12+\frac12\right)
\int_{\SE(3)}
f(\nabla_{\xi}\Phi\cdot v_3)\,d\mu_\Phi
=0.
\end{align*}
Together with Equation \eqref{TS}, we have $(\L f,1)_{\Hi}=0$.~Hence $\mu_\Phi$ is invariant for $(\L,\D)$.
\end{proof}

\begin{definition}
The Lemma \ref{lemma:lemma2.25,SE(3)} establishes that $(\S,\D)$, $(\A, \D)$ and $(\L, \D)$ are densely defined and dissipative in $\Hi$, and therefore are closable. The closures of $(\S,\D)$, $(\A, \D)$ and $(\L, \D)$ in $\Hi$ are denoted by $(\S, D(\S))$, $(\A, D(\A))$ and $(\L, D(\L))$, respectively. Consequently, $(\S, D(\S))$ is symmetric, while $(\A, D(\A))$ is antisymmetric. Moreover, one can show that $(\S,D(\S))$ is even self-adjoint. Since we do not need this fact in the sequel, we omit the proof. Very important for verifying the hypocoervivity assumptions \textnormal{\textbf{(H)}} in turn is the essential m-dissipativity of $(\L, \D)$, i.e., the m-dissipativity of $(\L, D(\L))$. This we prove in Section \ref{subsec:mdissi}, below.
\end{definition}

Next, we introduce two orthogonal projections that are important for the hypocoercive structure.

\begin{definition}\label{def:ortho-projection}
We define the orthogonal projection $\Pi_{\S}: \Hi \to \Hi$ by averaging over the rotation components
\begin{equation}\label{eq:projection_SE(3)}
  \Pi_{\S} f := \int_{\mathrm{SO}(3)} f \,\, d\nu_3,
  \quad f \in \Hi,
\end{equation}

where $\nu_3$ is the normalized Haar measure on $\SO(3)$.~ Furthermore, we define the orthogonal projection $\Pi:\Hi\to\Hi$ which additionally removes the component along the constant function $1$, as follows
\begin{equation}
  \Pi f := \Pi_{\S}f  - (f,1)_{\Hi}, \quad f \in \Hi.
\end{equation}
\end{definition}

The projections above are essential for the Poincaré-type inequalities needed in the Hypocoercivity Assumptions \textnormal{\textbf{(H1)}}-\textnormal{\textbf{(H4)}} (see Section \ref{Hypocoervicity_Assumptions_SE(n)}), and are used in the hypocoercive estimates.

\begin{lemma}\label{lemma:projection-domain}
The projections satisfy $\Pi(\Hi)\subset D(\S), \S\Pi f=0$,
as well as $\Pi(\D)\subset D(\A), \A\Pi(\D)\subset D(\A)$.~Moreover, for $f\in\D$, we have
\begin{equation}\label{2.21_SE(3)}
\begin{aligned}
 \A \P f &= - v_{3} \cdot \nabla_{\xi} f_{\S}, \quad \text{for all}  \quad f \in \D,
\end{aligned}
\end{equation}
where $f_{\S}:=\Pi_{\S} f \in C_c^{\infty}\left(\mathbb{R}^3\right)$.~
\end{lemma}

\begin{proof}
We divide the proof into two parts.

\begin{enumerate}

\item[(i)] We first prove that $\Pi(\Hi)\subset D(\S)$ and $\S\Pi=0$.~Note first that
\begin{equation}\label{eq:important-1}
\D\ni f
\mapsto
\Pi_{\S}f=f_{\S}
\in C_c^\infty(\mathbb{R}^3)
\subset\D.
\end{equation}
Thus, $\Pi_{\S} \D \subset \ker(\S)$.~We now extend this observation to the range of $\Pi$ in $\Hi$.~The range of $\Pi$ is identified with a subspace of $L^2(\mathbb{R}^3,\gamma_\Phi)$.~Let $h\in\Pi(\Hi)$.~Since $C_c^\infty(\mathbb{R}^3)$ is dense in $L^2(\mathbb{R}^3,\gamma_\Phi)$, there exists a sequence $(h_{n})_{n\in\mathbb{N}}\subset C_c^\infty(\mathbb{R}^3)$ such that
\begin{equation*}
h_{n}\rightarrow h
\,\, \text{in }L^2(\mathbb{R}^3,\gamma_\Phi).
\end{equation*}

Identifying $h_{n}$ and $h$ again with functions on $\SE(3)$ that are independent of the rotation variable,  we have $h_{n}\in\D$ and $\S h_{n}=0$, $n\in\mathbb{N}$.~Thus, $h_{n}\rightarrow h$, $\S h_{n}\rightarrow0$
in $\Hi$.~Since $(\S,D(\S))$ is closed, it follows that $h\in D(\S)$ and $\S h=0$.~Therefore,
\begin{equation*}
\Pi(\Hi)\subset D(\S),
\quad
\S\Pi=0.
\end{equation*}

\item[(ii)] We now prove the assertions concerning $\A$.~Let $f\in\D$.~Since $f_{\S}$ is independent of the rotation variables, we have $\widetilde{X}_i f_{\S}=0$, for $ i=1,2,3$.~Consequently,
\begin{equation}\label{eq:A-PiS-domain}
\begin{aligned}
\A f_{\S}=-v_{3}\cdot\nabla_{\xi} f_{\S}.
\end{aligned}
\end{equation}

It remains to take into account the constant appearing in the
definition of $\Pi$.~ We first show that $1\in D(\A)$ and $\A1=0$.~Let $(\varphi_{n})_{n\in\mathbb{N}}\subset
C_c^\infty(\mathbb{R}^3)$ be a sequence of cut-off functions satisfying
\begin{equation*}
0\leq\varphi_{n}\leq1,
\,
\varphi_{n}\rightarrow1
\,\, \text{pointwise},
\,
|\nabla_{\xi}\varphi_{n}|
\leq\frac{C}{n},
\end{equation*}
where $0 < C < \infty$ is independent of $n$.~We identify each $\varphi_{n}$ with
a function on $\SE(3)$ independent of the rotation variable. Then
$\varphi_{n}\in\D$, and by the same computation as in
\eqref{eq:A-PiS-domain}, $\A\varphi_{n}=-v_{3}\cdot\nabla_{\xi}\varphi_{n}$.~Since $|v_{3}|=1$,
$|\A\varphi_{n}|\leq
|\nabla_{\xi}\varphi_{n}|
\leq
\frac{C}{n}$.~As $\mu_\Phi$ is a probability measure, it follows that
\begin{equation*}
\|\A\varphi_{n}\|_{\Hi}
\leq
\frac{C}{n}
\rightarrow0.
\end{equation*}
Moreover, by dominated convergence, $\varphi_{n}\rightarrow 1$ in $\Hi$.~Since $(\A,D(\A))$ is closed, we conclude that $1\in D(\A)$, and $\A1=0$.~Now, by Definition~\ref{def:ortho-projection}, since $f_{\S}\in\D\subset D(\A)$ and $1\in D(\A)$, we obtain $\Pi f\in D(\A)$.~Furthermore,
\begin{align*}
\A\Pi f
=
\A f_{\S}-(f,1)_{\Hi}\A1
=
\A f_{\S}
&=
-v_{3}\cdot\nabla_{\xi} f_{\S}.
\end{align*}
This proves \eqref{2.21_SE(3)} and shows that
$\Pi(\D)\subset D(\A)$.~Finally, since $f_{\S}\in C_c^\infty(\mathbb{R}^3)$ and $v_{3}$ is smooth on $\SO(3)$, the function $-v_{3}\cdot\nabla_{\xi} f_{\S}$ belongs to $\D$.~Therefore, $\A\Pi f\in\D\subset D(\A)$, and hence $\A\Pi(\D)\subset D(\A)$.        
\end{enumerate}

This completes the proof.
\end{proof}

%%%%%%%%%%%%%%%%%%%%%%%%%%%%%%%%%%%%%%%%%%%%%%%%%%%%%%%%%%%%%%%%%%%
\begin{lemma}\label{lemma:4.20}
It holds that $1\in D(\L)$ and $\L 1=0$.~
\end{lemma}

\begin{proof}
Let $(\varphi_{n})_{n\in\mathbb N}\subset C_c^\infty(\mathbb R^3)$ be the sequence of cut-off functions used in the proof of Lemma~\ref{lemma:projection-domain}. Identifying $\varphi_{n}$ with a function on $\SE(3)$ which is independent of the rotation variable, we have $\varphi_{n}\in\D$ and $\S\varphi_{n}=0$.~Moreover, as shown in the proof of Lemma~\ref{lemma:projection-domain}, $\A \varphi_{n} = - v_{3}\cdot\nabla_{\xi}\varphi_{n}$.~Hence
\begin{equation*}
\L\varphi_{n} = \S\varphi_{n}- \A \varphi_{n}
= \,v_{3} \cdot\nabla_{\xi}\varphi_{n}.
\end{equation*}
Since $|v_{3}|=1$, and $|\nabla_{\xi}\varphi_{n}|\leq \frac{C}{n}$, we obtain
\begin{equation*}
\|\L\varphi_{n}\|_{\Hi}
\leq \frac{C}{n}
\rightarrow0.
\end{equation*}
Furthermore, $\varphi_{n}\to1$ in $\Hi$ by dominated convergence. Since $(\L,D(\L))$ is closed, it follows that $1\in D(\L)$ and  $\L1=0$.
\end{proof}

\subsection{Essential m-dissipativity of \texorpdfstring{$(\L,\D)$}{(L,D)}}\label{subsec:mdissi}

It is very convenient to check all hypocoercivity assumptions only on the set of nice functions 
$\D= C_c^\infty(\SE(3))$. However, this is only justified after showing that $\D$ is a core of an m-dissipative extension of $(\L,\D)$. In this context, the following theorem is of great importance. Although the operator $\L$ is degenerate, the specific construction of the antisymmetric operator $\A$ ensures that the associated vector fields satisfy Hörmander’s condition.

\begin{theorem}[Hörmander's condition]
\label{Hormanders-Condi}
Let $\sigma\in (0,\infty)$ and $\Phi\in C^\infty(\mathbb R^3)$. Define
 $Y_i:=\frac{\sigma}{\sqrt{2}}\,\widetilde{X}_i,\, i=1,2,3$, and
$ Y_0
    :=
    \widetilde{X}_6
    +\frac12(\nabla_\xi\Phi\cdot v_2)\widetilde{X}_1
    -\frac12(\nabla_\xi\Phi\cdot v_1)\widetilde{X}_2$.
Then
\begin{equation*}
    L=\sum_{i=1}^3Y_i^2+Y_0,
\end{equation*}
and the vector fields $Y_0,Y_1,Y_2,Y_3$ satisfy Hörmander's
bracket-generating condition, i.e.,
\begin{equation*}
    \operatorname{Lie}\{Y_0,Y_1,Y_2,Y_3\}(p)
    =T_p\SE(3) =\mathfrak{se}(3)
\end{equation*}
for every $p\in\SE(3)$. Consequently, $L+c$ is hypoelliptic for every $c\in C^\infty(\SE(3))$.
\end{theorem}

\begin{proof}
Since $Y_0$ differs from $\widetilde{X}_6$ only by smooth
linear combinations of the rotation vector fields
$\widetilde{X}_1$ and $\widetilde{X}_2$,
by the commutator relations in Equation~\eqref{eq:commutator-relations} we have
\begin{equation*}
[\widetilde{X}_1,Y_0]
=
-\widetilde{X}_5+R_1,
\qquad
[\widetilde{X}_2,Y_0]
=
\widetilde{X}_4+R_2,
\end{equation*}
where $R_1,R_2\in \operatorname{span}\{\widetilde{X}_1,\widetilde{X}_2,\widetilde{X}_3\}$. Hence, at every $p\in\SE(3)$,
\begin{equation*}
\widetilde{X}_1(p),\widetilde{X}_2(p),\widetilde{X}_3(p),
Y_0(p),
[\widetilde{X}_1,Y_0](p),
[\widetilde{X}_2,Y_0](p)
\end{equation*}
span the same tangent space as $\widetilde{X}_1(p),\ldots,\widetilde{X}_6(p)$. Therefore
\begin{equation*}
\operatorname{Lie}
\{\widetilde{X}_1,\widetilde{X}_2,\widetilde{X}_3,Y_0\}(p)
= T_p\SE(3) \,\, \text{for every} \,\, p\in\SE(3).
\end{equation*}

The last assertion follows from Hörmander's hypoellipticity theorem \cite{Hormander1967}.
\end{proof}

As a consequence of Theorem \ref{Hormanders-Condi} we obtain essential m-dissipativity:

\begin{theorem}\label{L-is-m-dissipative}
Let $\sigma \in (0, \infty)$ and let $\Phi$ satisfy the assumptions of Definition~\ref{def:measure_SE(3)}. Then $(\L,\D)$ is essentially m-dissipative on $\Hi$.~Consequently, its closure $(\L,D(\L))$ generates a strongly continuous contraction semigroup $(T_t)_{t\geq0}$ on $\Hi$.~
\end{theorem}

\begin{proof}
First, the dissipativity of the operator $(\L, \D)$ on $\Hi=L^2(\SE(3), \mu_{\Phi})$ follows directly from Lemma \ref{lemma:lemma2.25,SE(3)}. The operator $\L$ decomposes as $\L= \S - \A$, where $\S$ is symmetric and non-positive definite on $\Hi$, and $\A$ is antisymmetric and does not affect this property. By Lumer--Phillips, see \cite[Theorem 3.15]{Engel2010}, it remains to prove that $(I - L)(\D)$ is dense in $\Hi$. Let $f \in \Hi$ satisfy $((I-L)u,f)_{\Hi}=0$ for all $u\in \D$. We must show that $f=0$.

Note that $e^{-\Phi} f \in \L_{\text{Loc}}^{1}(\SE(3), \mu_{\Phi})$ for all $f \in \D$. Therefore, by Hörmander's condition, established in Theorem \ref{Hormanders-Condi}, and
\cite[Proposition~A.1]{GROT-STIL-2014}, it follows that
$e^{-\Phi}f\in C^\infty(\SE(3))$. Since $\Phi\in C^\infty(\mathbb R^3)$, we conclude that $f\in C^\infty(\SE(3))$. Since $\SO(3)$ is compact, it remains to apply a cut-off argument only in the non-compact variable $\xi$, as in Lemma \ref{lemma:projection-domain}.

Let $(\varphi_{n})_{n\in\mathbb{N}}\subset C_c^\infty(\mathbb{R}^3)$ be a sequence of cut-off functions such that $0\leq\varphi_{n}\leq1$ and $\varphi_{n}\rightarrow 1$ pointwise. There exists a constant $0 < C < \infty$, independent of $n$, such that $\sup |\nabla_{\xi}\varphi_{n}| \leq\frac{C}{n}$. Inspired by \cite[Proposition 5.5]{HN05}, we define $u_{n} := \varphi_{n}^{2} f$. Then,

\begin{equation}\label{summ}
\left(u_{n}, f\right)_{\Hi}=\left(L u_{n}, f\right)_{\Hi}=\left(S u_{n}, f\right)_{\Hi}-\left(A u_{n}, f\right)_{\Hi}.
\end{equation}

Since $(\S, \D)$ is symmetric and non-positive definite, and $\varphi_{n}$ is independent of the rotation variables, we have

\begin{equation}\label{sympart}
\begin{aligned}
\left(\S u_{n}, f\right)_{\Hi} = \left(\S \varphi_{n}^{2} f, f\right)_{\Hi} =  \left(\S (\varphi_{n} f), (\varphi_{n} f) \right)_{\Hi}  \leq 0.
\end{aligned}
\end{equation}

Since the antisymmetric operator $\A$ acts as a derivation, we have $\A u_{n}= \A (\varphi_{n} \varphi_{n} f)= (\A \varphi_{n}) (\varphi_{n} f) + \varphi_{n} \A (\varphi_{n} f)$. Therefore,

\begin{equation}\label{anti-1234}
\begin{aligned}
\left(A u_{n}, f\right)_{\Hi} &= \left((\A \varphi_{n}) (\varphi_{n} f), f\right)_{\Hi} + \left(\varphi_{n} \A (\varphi_{n} f), f\right)_{\Hi}\\
%&= - \left( (\varphi_{n} f), \A (\varphi_{n}) f\right)_{\Hi} + \left( \A (\varphi_{n} f), \varphi_{n} f\right)_{\Hi} \
&= \left( (\varphi_{n} f), \A (\varphi_{n}) f\right)_{\Hi} + \left( \A (\varphi_{n} f), \varphi_{n} f\right)_{\Hi}.
\end{aligned}
\end{equation}

The second term in Equation \eqref{anti-1234} cancels by the definition of the antisymmetric operator. Therefore,

\begin{equation}\label{anti-expression}
\begin{aligned}
\left(-A u_{n}, f\right)_{\Hi} %= - \left( (\varphi_{n} f), \A (\varphi_{n} f)\right)_{\Hi}
&=  -Z(\Phi)^{-1} \int_{\SE(3)} \A (\varphi_{n}) \varphi_{n} f^2  e^{-\Phi} d\xi d\nu_3\\
&= Z(\Phi)^{-1} \int_{\SE(3)} \left( v_{3}\cdot\nabla_{\xi}\varphi_{n} \right) \varphi_{n} f^2   e^{-\Phi} d\xi d\nu_3 \\
& \leq \frac{C}{n} Z(\Phi)^{-1} \int_{\SE(3)}  \varphi_{n} f^2  e^{-\Phi} d\xi d\nu_3 \\
& \leq \frac{C}{n} Z(\Phi)^{-1} \int_{\SE(3)}   f^2  e^{-\Phi} d\xi d\nu_3 = \frac{C}{n} \|f\|^{2}_{\Hi}.
\end{aligned}
\end{equation}

Hence for $n \rightarrow \infty$, \eqref{summ}, \eqref{sympart} and \eqref{anti-expression} together with Lebesgue dominated convergence imply $\|f\|_{\Hi}^{2} \leq 0$. Therefore, $f = 0$.
\end{proof}

\subsection{Verification of hypocoercivity assumptions (H)}

Next we specify the conditions on the potential $\Phi$ to ensure the hypocoercivity assumptions \textbf{(H)}. These read similarly to those in \cite{DKMS12}.

\begin{itemize}

\item[\textbf{(C1)}.] Let $\Phi: \mathbb{R}^3 \to \mathbb{R}$ be bounded from below, such that $\Phi \in C^{\infty}(\mathbb{R}^3)$, and $\gamma_{\Phi}$ is a probability measure on $(\mathbb{R}^3, \Bor(\mathbb{R}^3))$.~

\item[\textbf{(C2)}.]\label{cond:C2_SE(3)} There exists $\Lambda \in (0, \infty)$ such that the probability measure $\gamma_{\Phi}$ satisfies the Poincaré inequality of the form
\begin{equation*}
\|\nabla_{\xi} f\|_{L^2\left(\gamma_{\Phi} \right)}^2 \geq \Lambda\left\|f-(f, 1)_{L^2\left(\gamma_{\Phi} \right)}\right\|_{L^2\left(\gamma_{\Phi} \right)}^2 \quad \text{for all} \quad f \in C_c^{\infty}(\mathbb{R}^3). 
\end{equation*}

\item[\textbf{(C3)}.] There exists a constant $c < \infty$ such that
\begin{equation*}
|\nabla_{\xi}^2 \Phi(\xi)| \leq c (1+ |\nabla_{\xi} \Phi(\xi)|) \quad \text{for all} \quad \xi \in \mathbb{R}^3.
\end{equation*}

\end{itemize}

\begin{remark}

The Assumptions \textnormal{\textbf{(C1)}}--\textnormal{\textbf{(C3)}} enter the verification of the hypocoercivity conditions in different ways. Condition \textnormal{\textbf{(C2)}} provides a Poincaré inequality for the measure $\gamma_{\Phi}$, which is the key ingredient for the macroscopic coercivity estimate \textnormal{\textbf{(H3)}}. For sufficient conditions implying \textnormal{\textbf{(C2)}}, we refer to \cite{wang2013diffusion,BakryBartheCattiauxGullin2008}. Finally, the elliptic a priori estimates required to verify Assumption \textnormal{\textbf{(H4)}} rely on the combination of conditions \textnormal{\textbf{(C2)}} and \textnormal{\textbf{(C3)}}.
\end{remark}

Throughout this subsection, we assume that the potential $\Phi$ satisfies conditions \textnormal{\textbf{(C1)}}--\textnormal{\textbf{(C3)}}. We work in the Hilbert space $\Hi=L^2\left(\SE(3), \mu_{\Phi}\right)$, with core $\D=C_c^{\infty}(\SE(3))$, operator $\L=\S-\A$, and projections $\Pi_S$ and $\Pi$ defined in Definition \ref{def:ortho-projection}. We now verify the abstract hypocoercivity Assumptions \textnormal{\textbf{(H1)}}--\textnormal{\textbf{(H4)}} from Section \ref{The_Hilbert_Method_Space} for the operator $\L$ on $\SE(3)$.

\begin{proposition}[Assumption \textnormal{\textbf{(H1)}}]\label{prop:H1} 

The algebraic relation
\begin{equation}
  \Pi \A \Pi = 0 \quad \text{on} \,\, \D.
\end{equation}
\end{proposition}

\begin{proof}
Let $ f\in \D$.~Since
\begin{equation*}
\Pi f=\Pi_{\S} f-(f,1)_{\Hi},
\end{equation*}
and $\A 1=0$, we obtain $\A \Pi f=\A \Pi_{\S} f$.~From Lemma \ref{lemma:lemma2.25,SE(3)}, we have that the term $\A \Pi f= - v_{3} \cdot \nabla_{\xi} g$ where $g= \Pi_{\S} f$, which in coordinates is given by 
\begin{equation*}
    \A \Pi f =- ( s_\alpha s_\beta \nabla_{\xi_1} f_{\S} -c_\alpha s_\beta \nabla_{\xi_2} f_{\S} + c_\beta \nabla_{\xi_3} f_{\S}). 
\end{equation*}

Let us note that the coefficients depend on the Euler angles $\alpha, \beta$ and $\gamma$, hence, if we apply the orthogonal projection $\Pi_{\S}$ in Definition \ref{def:ortho-projection}, we have the following
\begin{equation*}
    \begin{aligned}
&\Pi_{\S} \A \Pi f %=\int_{\SO(3)}\left(- v_{3} \cdot \nabla_{\xi} f_{\S} \right) \,\, d \nu_3 \\
=-\int_{\SO(3)}\left( s_\alpha s_\beta \nabla_{\xi_1} f_{\S} -c_\alpha s_\beta \nabla_{\xi_2} f_{\S} + c_\beta \nabla_{\xi_3} f_{\S} \right) \,\, d\nu_3\\
&=- \int_{\SO(3)}\left( s_\alpha s_\beta \nabla_{\xi_1} f_{\S} \right) \,\, d\nu_3 + \int_{\SO(3)}\left( c_\alpha s_\beta \nabla_{\xi_2} f_{\S} \right) \,\, d\nu_3 - \int_{\SO(3)}\left( c_\beta \nabla_{\xi_3} f_{\S} \right) \,\, d\nu_3 \\
&= -\underbrace{\int_{\SO(3)} s_\alpha s_\beta \,\, d\nu_3}_\text{=0} \cdot \nabla_{\xi_1} f_{\S}  + \underbrace{\int_{\SO(3)}  c_\alpha s_\beta  \,\, d\nu_3 }_\text{=0} \cdot \nabla_{\xi_2} f_{\S} - \underbrace{\int_{\SO(3)} c_\beta \,\, d\nu_3 }_\text{=0} \cdot \nabla_{\xi_3} f_{\S}=0.
    \end{aligned}
\end{equation*}

This shows $\Pi_{\S} \A \Pi f=0$ for all $f\in \D$.~Moreover, since $\A$ is antisymmetric, we have that
\begin{equation*}
(\A \Pi f,1)_{\Hi} = -(\Pi f,\A1)_{\Hi} = 0.
\end{equation*}

Therefore,
\begin{equation*}
\begin{aligned}
\Pi \A \Pi f = \Pi_S \A \Pi f-(\A \Pi f,1)_{\Hi} =0 \quad\text{on} \quad \mathcal{D}.
\end{aligned}
 \qedhere
 \end{equation*}
\end{proof}

\begin{remark}
The vanishing of the above integrals is a direct consequence of the properties of the normalized Haar measure on $\SO(3)$ by rotation invariance, i.e., the first-order spherical harmonics (such as $s_\alpha s_\beta$, $c_\alpha s_\beta$, $c_\beta$) have zero average over $\SO(3)$.~
\end{remark}

%------------------------------------------------------------------
\begin{proposition}[Assumption \textnormal{\textbf{(H2)}}, microscopic coercivity]\label{Assumption-H2}  There exists $\Lambda_m >0$ such that 
\begin{equation}\label{3.14}
    -(\S f, f)_{\Hi} \geq \Lambda_m \|(\I- \Pi_{\S})f\|^2_{\Hi} \quad \text{for all} \quad f \in \D.
\end{equation}
In fact, one may take $\Lambda_m=\sigma^2$.~
\end{proposition}

\begin{proof}
From the Rayleight-Ritz characterization on the compact Lie group $\SO(3)$, see \cite[Chapter I, Section 5]{Chavel1984}, we have the following Poincaré inequality on $\SO(3)$:
\begin{equation}\label{eq:Poinca_SO(3)}
     %-(\S f,f)_{L^2(\nu_{3})} 
     %= -\left(\frac{\sigma^2}{2} \Delta_{\SO(3)}f,f\right)_{L^2(\mu_{\SO(3)})}
    \|\nabla_{\SO(3)} h\|_{L^2(\nu_{3})}^2 
     %\geq                       \quad \text{for} \, f \in C^{\infty}(\SO(3)),
     \geq \lambda_{1} \|(I-\Pi_{\S})h\|_{L^2(\nu_{3})}^2 \quad \text{for}  \,\, h \in C^{\infty}(\SO(3)),
    % = \sigma^2 \|(I-\Pi_{\S})f\|^2,
\end{equation}

where $\lambda_1$ is the first non-zero eigenvalue of $- \Delta_{\SO(3)}$.~By \cite[Theorem 1.2]{Lauret_2018}, we have $\lambda_1=2$.~Let $f \in \D$.~For each fixed $\xi \in \mathbb{R}^{3}$, let us consider $h_\xi: \SO(3)\rightarrow\mathbb{R}$, $h_\xi(\cdot):=f(\xi,\cdot)$. Then \eqref{eq:Poinca_SO(3)} applied to $h_\xi$ and integrated with respect to $\gamma_{\Phi}$ yields
\begin{equation*}
-(\S f,f)_{\Hi} \geq \sigma^2\|(\I-\Pi_{\S})f\|_{\Hi}^2 \quad \text{for} \,\, f \in \D.
\end{equation*}
Therefore, the microscopic coercivity inequality \eqref{3.14} holds with
$\Lambda_m=\sigma^2$.~
\end{proof}

To verify Assumption \textnormal{\textbf{(H3)}}, we compute the operator $ \A^2 \Pi$ and $G:= \Pi \A^2 \Pi$ on the core $\D$ in the following lemma.

%-----------------------------------------------------
%\newpage

\begin{lemma}\label{lemma:4.19}
Let $f\in \D$.~Then the following expression holds
\begin{equation}\label{2.22_SE(3)}
\begin{aligned}
\A^2 \P f &= v_{3}^{\top} \, \nabla_{\xi}^2 f_{\S} \,  v_{3}  - \frac{1}{2}   (\nabla_{\xi} \Phi)^{\top} \, (I - v_{3} \otimes v_{3}) \, \nabla_{\xi} f_{\S}.
\end{aligned}
\end{equation}  
\end{lemma}

\begin{proof}
Let us compute $\A^2 \Pi$ as follows
\begin{equation}\label{eq:AAPf_term_SE(3)}
    \begin{aligned}
        &\A^2 \Pi f  = - \left(  - \widetilde{X}_6   -  \frac{1}{2}  \left( \nabla_{\xi} \Phi \cdot v_{2} \,\, \widetilde{X}_{1} - \nabla_{\xi} \Phi \cdot v_{1} \,\, \widetilde{X}_{2} \right) \right) ( v_{3} \cdot \nabla_{\xi} f_{\S} )  \\
        &=  \underbrace{ \widetilde{X}_6 \left( v_{3} \cdot \nabla_{\xi} f_{\S} \right) }_\text{TI} +  \underbrace{\frac{1}{2}  \left( (\nabla_{\xi} \Phi \cdot v_{2}) \widetilde{X}_{1} \left(  v_{3} \cdot \nabla_{\xi} f_{\S}  \right) - (\nabla_{\xi} \Phi \cdot v_{1}) \widetilde{X}_{2} \left( v_{3} \cdot \nabla_{\xi} f_{\S}  \right) \right)}_\text{TII}.\\
        %&= \underbrace{(v_{3} \otimes v_{3}) \nabla_{\xi}^2 g(\xi)}_\text{TI} - \underbrace{\frac{1}{2}  \left( (X_6^r(\nabla_{\xi} \Phi(\xi) \cdot v_{2})) \widetilde{X}_{1}^r - X_6^r((\nabla_{\xi} \Phi(\xi) \cdot v_{1})) \widetilde{X}_{2}^r \right)}_\text{TII}.
    \end{aligned}
\end{equation}

%\improvement{Not bold letters for TI or TII, where I wrot}
%%%%%%%%%%
\textbf{Term TI}. %\improvement{Here I should add a small remark on the proof of TI} 
To compute TI, we proceed as follows
\begin{equation}\label{termI-}
    \begin{aligned}
       & \widetilde{X}_6^r \left(  v_{3} \cdot \nabla_{\xi} f_{\S} \right) = (v_{3} \cdot \nabla_{\xi}) (v_{3} \cdot \nabla_{\xi} f_{\S} ) = \sum_{i=1}^3 v_{3,i} \partial_{\xi_i} \left( \sum_{j=1}^3 v_{3,j} \partial_{\xi_j} f_{\S}  \right) \\
        &=  \sum_{i,j=1}^3 v_{3,i}  v_{3,j} \, \partial_{\xi_i}  \partial_{\xi_j} f_{\S} = \sum_{i,j=1}^3 v_{3,i} v_{3,j} \partial^2_{\xi_i \xi_j}f_{\S} =v_{3}^{\top} \, \nabla_{\xi}^2 f_{\S} \,  v_{3}.
        %= (v_{3} \otimes v_{3}) \, \nabla_{\xi}^2 g(\xi). 
    \end{aligned}
\end{equation}

\textbf{Term TII}. We now compute TII. Recall that the vector fields $\widetilde{X}_1, \widetilde{X}_2$, and $\widetilde{X}_3$ are the rotation ones from the Lie group $\SO(3)$, and therefore independent of the variable $\xi$.~Consequently, their action affects only the coefficients of $ v_{3} = \widetilde{X}_6$.~We need to consider only the following contributions given below 
%\footnote{$\widetilde{X}_{1} \left( v_{3} \cdot \nabla_{\xi} f_{\S}(\xi)  \right)= \sum_{j=1}^3 \widetilde{X}_{1} (v_{3,j}) \partial_{\xi_{j}} f_{\S} (\xi)$} 
\begin{equation*}
    \begin{aligned}
      \widetilde{X}_{1} \left(  v_{3} \cdot \nabla_{\xi} f_{\S}  \right) &=   \widetilde{X}_{1} \left( v_{3}  \right) \cdot \nabla_{\xi} f_{\S} \\
      &=  \left(  \frac{s_{\gamma}}{s_{\beta}} c_{\alpha} s_{\beta} + c_{\gamma} s_{\alpha} c_{\beta}, \frac{s_{\gamma}}{s_{\beta}} s_{\alpha} s_{\beta} - c_{\gamma} c_{\alpha} c_{\beta}, -s_\beta c_{\gamma} \right)   \cdot \nabla_{\xi} f_{\S} \\
      &= ( s_{\gamma} c_{\alpha}  + c_{\gamma} s_{\alpha} c_{\beta} , s_{\gamma} s_{\alpha} - c_{\gamma} c_{\alpha} c_{\beta} , -s_\beta c_{\gamma}  )   \cdot \nabla_{\xi} f_{\S} \\   
      &=- v_{2}  \cdot \nabla_{\xi} f_{\S},\\
      \widetilde{X}_{2} \left(  v_{3} \cdot \nabla_{\xi} f_{\S}  \right) %\widetilde{X}_{2}^r \left( v_{3}  \right) \cdot \nabla_{\xi} g(\xi) 
       &= \left( \frac{c_\gamma}{s_\beta} c_\alpha s_\beta - s_\gamma s_\alpha c_\beta,  \frac{c_\gamma}{s_\beta} s_\alpha s_\beta - s_\gamma (-c_\alpha c_\beta) , -s_\gamma (-c_\beta)   \right)  \cdot \nabla_{\xi} f_{\S} \\
       &= ( c_\gamma c_\alpha - s_\gamma s_\alpha c_\beta,  c_\gamma s_\alpha  + s_\gamma c_\alpha c_\beta , s_\gamma c_\beta   )  \cdot \nabla_{\xi} f_{\S} \\
       &=  v_{1} \cdot \nabla_{\xi} f_{\S}.
    \end{aligned}
\end{equation*}
Therefore, 
\begin{equation}\label{termII-}
    \begin{aligned}
    \text{TII}&= \frac{1}{2}  \left( (\nabla_{\xi} \Phi \cdot v_{2}) \widetilde{X}_{1} \left( v_{3} \cdot \nabla_{\xi} f_{\S}  \right) - (\nabla_{\xi} \Phi \cdot v_{1}) \widetilde{X}_{2} \left( v_{3} \cdot \nabla_{\xi} f_{\S}  \right) \right)\\
    & =  \frac{1}{2}  \left( (\nabla_{\xi} \Phi \cdot v_{2})  \left( - v_{2} \cdot \nabla_{\xi} f_{\S}  \right) - (\nabla_{\xi} \Phi \cdot v_{1}) \left( v_{1} \cdot \nabla_{\xi} f_{\S}  \right) \right)  \\
    &= - \frac{1}{2}  \left( (\nabla_{\xi} \Phi \cdot v_{2})  \left( v_{2} \cdot \nabla_{\xi} f_{\S}  \right) + (\nabla_{\xi} \Phi \cdot v_{1}) \left( v_{1} \cdot \nabla_{\xi} f_{\S}  \right) \right)\\
    &= - \frac{1}{2}  \left((\nabla_{\xi} \Phi)^{\top} \, (v_{2} \otimes v_{2}) \, \nabla_{\xi} f_{\S}   + (\nabla_{\xi} \Phi)^{\top} \, (v_{1} \otimes v_{1}) \, \nabla_{\xi} f_{\S}  \right)\\
    &= - \frac{1}{2} (\nabla_{\xi} \Phi)^{\top} \, (I - v_{3} \otimes v_{3}) \, (\nabla_{\xi} f_{\S} ).
    %&= \frac{1}{2}  \left( I - v_{3}^{\top} \cdot (\nabla_{\xi} \Phi(\xi) \otimes \nabla_{\xi} g(\xi)) \cdot v_{3}  \right)\\
    %&= \frac{1}{2}(\nabla_{\xi} g) \left( v_{1}^{\top} \, \nabla_{\xi}^2 g(\xi) \,  v_{1} + v_{2}^{\top} \, \nabla_{\xi}^2 g(\xi) \,  v_{2} \right) (\nabla_{\xi} \Phi) 
    %&= \frac{1}{2}(\nabla_{\xi} g) \left( I_3 - v_{3} \otimes v_{3} \right) (\nabla_{\xi} \Phi) \\ 
    \end{aligned}
\end{equation}

Hence, adding the terms \text{TI} and \text{TII} yields the desired result in Equation \eqref{2.22_SE(3)}.
\end{proof}

\begin{lemma}\label{lem:Auxiliary_Lemma_Computations}
Let $M(\xi)=\left(a_{ij}(\xi) \right)_{i,j=1}^3$ be a $3\times3$ matrix depending only on
$\xi\in\mathbb{R}^3$.~ Then
\begin{equation}
    \begin{aligned}
      \Pi_{\S} \left(v_{3}^{\top}  M  v_{3}\right) =
\frac{1}{3}\operatorname{tr}(M).
    \end{aligned}
\end{equation}
In particular, taking $M=I$, one has $\Pi_{\S} \left(v_{3}^{\top}Iv_{3}\right)=\frac{1}{3
}\operatorname{tr}(I)=1$.~  
\end{lemma}

\begin{proof}
Let $M(\xi)=\bigl(a_{ij}(\xi)\bigr)$ for $i,j=1,\cdots,3$ be a $3\times3$ matrix depending only on $\xi\in\mathbb R^3$.~ Since $M$ depends only on $\xi$, its coefficients are independent of the rotation variables $R$ and therefore commute with the projection operator $\Pi$, as follows
\begin{equation*}
    \begin{aligned}
    \Pi_{\S} \left(v_{3}^{\top} M v_{3}\right) & =\Pi_{\S} \left(\sum_{i, j=1}^3 a_{i j} v_{3, i} v_{3, j}\right) = \sum_{i, j=1}^3 a_{i j} \Pi_{\S} \left(v_{3, i} v_{3, j}\right)  =\sum_{i=1}^3 a_{i i} \frac{1}{3}  =\frac{1}{3} \operatorname{tr}(M),  \end{aligned}
\end{equation*}

where we used that $\Pi_{\S} \left(v_{3, i} v_{3, j}\right)=\frac{1}{3} \delta_{i j}$, i.e., only the diagonal terms survive as follows
\begin{equation}\label{integral-SO(3)}
\begin{aligned}
\int_{\SO(3)} s_\alpha^2 s_\beta^2\,d\nu_3 = \frac{1}{3},
\,
\int_{\SO(3)} c_\alpha^2 s_\beta^2\,d\nu_3 = \frac{1}{3},
\int_{\SO(3)} c_\beta^2\,d\nu_3 = \frac{1}{3},
\,
\int_{\SO(3)} s_\alpha c_\alpha s_\beta^2\,d\nu_3 = 0.
\end{aligned}
\end{equation}
The remaining off-diagonal terms vanish analogously.
\end{proof}

%%%%%%%%%%%%%%%%%%%%%%%%%%%%%%%%%%%%%%%%%%%%%%%%%%%%%%%%%%%%%%%%%%

\begin{proposition}\label{prop:auxiliary1-assump(H3)}
The operator $G:=\Pi\A^2\Pi$ satisfies, for every $f\in\D$,

\begin{equation}\label{eq:def-operator-G}
\begin{aligned}
Gf = \Pi \A^2 \Pi f &=  \frac{1}{3} \left( \Delta_{\xi} \, f_{\S} - \nabla_{\xi} \Phi \cdot \nabla_{\xi} \, f_{\S} \right). 
\end{aligned}
\end{equation}
\end{proposition}

\begin{proof}

Let $f\in \D$.~Note that $\Pi f=\Pi_{\S} f-(f,1)_{\Hi}=f_{\S}-(f,1)_{\Hi}$.~Furthermore, since $(f,1)_{\Hi}$ is a constant, we have $\A(f,1)_{\Hi}=0$.~Hence
\begin{equation*}
\A\Pi f=\A\Pi_{\S} f,
\quad
\A^2\Pi f=\A^2\Pi_{\S} f.
\end{equation*}

Therefore, in the computation of $\A^2\Pi f$, we decompose $\A^2\Pi f$ into the two terms obtained in the proof of Lemma \ref{lemma:4.19}, and apply the orthogonal projection in Definition \ref{def:ortho-projection} to each term separately. Note that $\nabla_{\xi} f_{\S}(\xi)$ and $\nabla_{\xi}^2 f_{\S}(\xi)$ are independent of the rotation variables $\alpha,\beta,\gamma$. 
%\improvement{Lemma 3.8, $\Pi_S TI$}

\begin{itemize}
    \item[\textbf{TI.}] Let us first compute the projection onto the first term TI in \eqref{termI-} as follows
\begin{equation}
    \begin{aligned}
       & \Pi_{\S} (\text{TI}) = \Pi_{\S} (v_{3}^{\top} \, \nabla_{\xi}^2 f_{\S} \, v_{3})
 = \frac{1}{3} \, \Delta_{\xi} f_{\S}, 
    \end{aligned}
\end{equation}

where the last expression was obtained using $M= \nabla_{\xi}^2 f_{\S}$ in Lemma \ref{lem:Auxiliary_Lemma_Computations}.

 \item[\textbf{TII.}]  For the term TII \eqref{termII-}, we apply the orthogonal projection \eqref{eq:projection_SE(3)} to Equation \eqref{eq:AAPf_term_SE(3)}. Note that the only terms that depend on the Euler angles are the vectors $v_{1}$ and $v_{2}$.~Hence,

\begin{equation*}\label{pi_v3_otimes_v3}
    \begin{aligned}
        \Pi_{\S} (\text{TII}) %&= \Pi_{\S} \left( - \frac{1}{2}   (\nabla_{\xi} \Phi)^{\top} \, (I - v_{3} \otimes v_{3}) \, (\nabla_{\xi} f_{\S})   \right)   \\
        %&= - \frac{1}{2}   (\nabla_{\xi} \Phi)^{\top} \, \Pi_{\S} \left( I - v_{3} \otimes v_{3} \right) \, (\nabla_{\xi} f_{\S}),  
        &= - \frac{1}{2}   (\nabla_{\xi} \Phi)^{\top} \, \left( \int_{\mathrm{SO}(3)} \left( I - v_{3} \otimes v_{3} \right) \,\, d\nu_3  \right) \, (\nabla_{\xi} f_{\S}) \\
        &= - \frac{1}{2}   (\nabla_{\xi} \Phi)^{\top} \, \left( \frac{2}{3} I \right) \, (\nabla_{\xi} f_{\S}).
    \end{aligned}
\end{equation*}

This follows from computations similar to those in the proof of Lemma \ref{lem:Auxiliary_Lemma_Computations}, see \eqref{integral-SO(3)}. Therefore,

\begin{equation*}
    \begin{aligned}
        \Pi_{\S} (\text{TII}) 
        %&= -\frac{1}{2}  (\nabla_{\xi} \Phi)^{\top} \, \left( \frac{2}{3} I \right)  \, (\nabla_{\xi} f_{\S}) 
        %= -\frac{1}{3} \, \nabla_{\xi} \Phi  \cdot \nabla_{\xi} f_{\S}.
        = -\frac{1}{3} (\nabla_{\xi} \Phi)^{\top}  \cdot (\nabla_{\xi} f_{\S}).
    \end{aligned}
\end{equation*}

\end{itemize}

Combining the above expressions, we have 

\begin{equation*}
    \Pi_{\S} \A^2 \Pi f = \frac{1}{3} \left( \Delta_{\xi} f_{\S}  - \nabla_{\xi} \Phi \cdot \nabla_{\xi} f_{\S} \right).
\end{equation*}

To complete the proof, we show that for all $f \in \D$, $(\Pi_{\S} \A^2 \Pi f,1)_{\Hi}=0$.~Using the integrant independence of the rotation variable, we obtain the following result
\begin{equation*}
    \begin{aligned}
      &(\Pi_{\S} \A^2 \Pi f,1)_{\Hi} = \int_{\SE(3)} \frac{1}{3} \left( \Delta_{\xi} f_{\S}  - \nabla_{\xi} \Phi \cdot \nabla_{\xi} f_{\S} \right) \,   d \mu_\Phi \\
      &=  \int_{\mathbb{R}^3} \frac{1}{3} \left( \Delta_{\xi} f_{\S}  \right) \,   \,\, d \gamma_\Phi -  \int_{\mathbb{R}^3} \frac{1}{3} \left( \nabla_{\xi} \Phi \cdot \nabla_{\xi} f_{\S} \right)  \,\, d \gamma_\Phi\\
      &=  Z(\Phi)^{-1} \int_{\mathbb{R}^3} \frac{1}{3} \left( \nabla_{\xi} \Phi \cdot \nabla_{\xi} f_{\S} \right)  \,\, e^{-\Phi} d \xi -  Z(\Phi)^{-1}  \int_{\mathbb{R}^3} \frac{1}{3} \left( \nabla_{\xi} \Phi \cdot \nabla_{\xi} f_{\S} \right)  \,\, e^{-\Phi} d \xi=0. 
    \end{aligned}
\end{equation*}

Therefore, we conclude that for all $f \in \D$
\begin{equation*}
\begin{aligned}
Gf = \Pi \A^2 \Pi f &=  \frac{1}{3} \left( \Delta_{\xi} \, f_{\S}  - \nabla_{\xi} \Phi \cdot \nabla_{\xi} \, f_{\S}\right).
\end{aligned}
 \qedhere
\end{equation*}

\end{proof}

The following proposition is a special case of \cite[Proposition 3.13]{GROT-STIL-2014}.

\begin{proposition}\label{G-mdissipative}
Let condition \textnormal{\textbf{(C1)}} be fulfilled. Then $(\I - \Pi \A^2 \Pi)(\mathcal{D})$ is dense in $\Hi$.~Consequently, $(G,\mathcal{D})$ is essentially m-dissipative in $\Hi$.
\end{proposition}

We can now prove Assumption \textnormal{\textbf{(H3)}}.

%There exists $\Lambda_m >0$ such that 

\begin{proposition}[Assumption \textbf{(H3)}, Macroscopic coercivity]\label{Macros_coerci}
%Let $\Phi: \mathbb{R}^3 \to \mathbb{R}$ satisfies conditions \textnormal{\textbf{(C1)}}, \textnormal{\textbf{(C2)}}, and 
There exists $\Lambda_M>0$ such that
\begin{equation}
    \| \A \Pi f \|_{\Hi}^2 \geq \Lambda_M \|\Pi f\|_{\Hi}^2 \quad \text{for all} \, f \in D((\A \Pi)^*(\A \Pi)).
\end{equation}

In particular, $\Lambda_M=\frac{\Lambda}{3}$.~
\end{proposition}

\begin{proof}
Let $f \in \D$.~Using Equation \eqref{eq:def-operator-G}, an integration by parts with respect to the probability measure $\gamma_{\Phi}$ together with its Poincaré inequality from \textnormal{\textbf{(C2)}} yields:

\begin{equation*}
    \begin{aligned}
        \|\A \Pi f\|_{\Hi}^2 &= -(G f,f)_{\Hi} =  \frac{1}{3}  \int_{\mathbb{R}^3}\left|\nabla_{\xi} f_{\S} \right|^2 \,\, d \gamma_{\Phi} \\
        %&= Z(\Phi)^{-1} \int_{\mathbb{R}^3} \int_{\SO(3)} (\nabla_{\xi} g)^{\top} \cdot \left(v_{3} \otimes v_{3}\right) \cdot (\nabla_{\xi} g)  \,\, \nu_3 \,\, e^{-\Phi(\xi)} d \xi
        \\
        %&= Z(\Phi)^{-1} \int_{\mathbb{R}^3} (\nabla_{\xi} g)^{\top}\left(\int_{ \SO(3)} v_{3} \otimes v_{3} \,\, \nu_3 \right) (\nabla_{\xi} g)  \,\, e^{-\Phi(\xi)} d \xi  \\
        %&= \int_{\SE(3)}  \left( v_{3} \cdot \nabla_{\xi} g \right)^2  Z(\Phi)^{-1}  e^{-\Phi(\xi)} d \xi \otimes \mu_{\SO(3)}\\
        %&= Z(\Phi)^{-1} \int_{\mathbb{R}^3} \left( \frac{1}{8 \pi^2} \int_{\SO(3)} v_{3} \otimes v_{3}  \, \mu_{\SO(3)} \right) \quad |\nabla_{\xi} g|^2 e^{-\Phi(\xi)} d \xi = \frac{1}{3}  Z(\Phi)^{-1} \int_{\mathbb{R}^3}  |\nabla_{\xi} g|^2  e^{-\Phi(\xi)} d \xi\\
        & \overset{\textbf{(C2)}}{\geq} \frac{\Lambda}{3}  \int_{\mathbb{R}^3} \left( f_{\S} -  \int_{\mathbb{R}^3} f_{\S} \,\,  d \gamma_{\Phi} \ \right)^2 \,\,  d \gamma_{\Phi} = \frac{\Lambda}{3} \, \|\Pi f\|_{\Hi}^2.
    \end{aligned}
\end{equation*}

By Proposition \ref{G-mdissipative}, the operator $(G, \D)$ is essentially m-dissipative. Therefore, Corollary \ref{cor:cor2.15P} applies and extends the estimate from the core $\D$ to $D\bigl((\A \Pi)^* \A \Pi\bigr)$.~ Hence, Assumption \textnormal{(H3)} follows with $\Lambda_M=\frac{\Lambda}{3}$.~
\end{proof}

To verify Assumption \textnormal{\textbf{(H4)}}, we treat the two auxiliary operator bounds separately. For the first bound, we apply Lemma~\ref{auxiliary_lemma_1}, following the abstract hypocoercivity method taking into account domain issues developed in \cite{GROT-STIL-2014}. For the second bound, we use an elliptic regularity estimate from \cite{DMS15}, which relies on conditions \textnormal{\textbf{(C1)}}--\textnormal{\textbf{(C3)}}.

\begin{proposition}[Assumption \textbf{(H4)}]\label{auxiliary_lemma+assumption(H4)_SE(3)}

The auxiliary operator bounds in \textnormal{\textbf{(H4)}} hold with
\begin{equation*}
c_1=\frac{1}{2}|c_3|=\frac{\sigma^2}{2},
\quad
c_2=c_\Phi,
\end{equation*}
where $c_3 = - (3-1)\frac{\sigma^2}{2}$ arises from the Laplace--Beltrami eigenvalue of the components of $v_3$, and $c_\Phi\in[0,\infty)$ depends on the potential $\Phi$ only.~
\end{proposition}

\begin{proof}

%\begin{enumerate}
    %\item 
    (i) To prove the first inequality in \textnormal{\textbf{(H4)}},
    see Equation~\eqref{aucxill}, we apply
    Lemma~\ref{auxiliary_lemma_1}. Thus, it is enough to
    determine a constant $c_3\in\mathbb{R}$ such that
    \begin{equation*}
        (\S h,\A\Pi f)_{\Hi}
        =
        c_3(h,\A\Pi f)_{\Hi},
        \quad f,h\in\D.
    \end{equation*} 
    
Using the explicit expression of the Laplace--Beltrami operator on
$\SO(3)$, a direct computation gives
\begin{equation}\label{Laplacian-vector-v3}
\begin{aligned}
\Delta_{\SO(3)}(s_\alpha s_\beta)
    &=-(3-1)s_\alpha s_\beta,\\
\Delta_{\SO(3)}(c_\alpha s_\beta)
    &=-(3-1)c_\alpha s_\beta,\\
\Delta_{\SO(3)}(c_\beta)
    &=-(3-1)c_\beta.
\end{aligned}
\end{equation}

Equivalently, since $v_3=(s_\alpha s_\beta,-c_\alpha s_\beta,c_\beta)^\top$,
Equation~\eqref{Laplacian-vector-v3} can be written as
\begin{equation*}
    \Delta_{\SO(3)}v_3=-(3-1)v_3.
\end{equation*}

We apply the identities above to evaluate $(\S h,\A\Pi f)_{\Hi}$.~ Let $f,h \in \D$, then
\begin{equation*}
    \begin{aligned}
        &(\S h, \A \Pi f )_{\Hi} %= \int_{\SE(3)} \S h \, (\A \Pi f) \, d \mu_\Phi\\
        =\int_{\SE(3)} \frac{\sigma^2}{2} \Delta_{\SO(3)} h \, \Big(s_\alpha s_\beta \nabla_{\xi_1} f_{\S} -c_\alpha s_\beta \nabla_{\xi_2} f_{\S} + c_\beta \nabla_{\xi_3} f_{\S} \Big) \, d \mu_\Phi\\
        &= - \frac{\sigma^2}{2}  \int_{\SE(3)} \nabla_{\SO(3)}  h \cdot \nabla_{\SO(3)} \Big(  s_\alpha s_\beta \nabla_{\xi_1} f_{\S} -c_\alpha s_\beta \nabla_{\xi_2} f_{\S} + c_\beta \nabla_{\xi_3} f_{\S} \Big) \, d \mu_\Phi\\
        %&=  \frac{\sigma^2}{2}  \int_{\SE(3)}  h \, \Delta_{\SO(3)} \Big(  s_\alpha s_\beta \nabla_{\xi_1} f_{\S} -c_\alpha s_\beta \nabla_{\xi_2} f_{\S} + c_\beta \nabla_{\xi_3} f_{\S} \Big) \, d \mu_\Phi \\
        &=  \frac{\sigma^2}{2}  \int_{\SE(3)}  h \, \Big( \Delta_{\SO(3)} (  s_\alpha s_\beta )\nabla_{\xi_1} f_{\S} - \Delta_{\SO(3)} (c_\alpha s_\beta
        )\nabla_{\xi_2} f_{\S} + \Delta_{\SO(3)}  (c_\beta )\nabla_{\xi_3} f_{\S} \Big) \, d \mu_\Phi \\
         %&=  \frac{\sigma^2}{2} (-2)  \int_{\SE(3)} h \,  \Big(  s_\alpha s_\beta \nabla_{\xi_1} f_{\S} -  c_\alpha s_\beta
        %\nabla_{\xi_2} f_{\S} +   c_\beta \nabla_{\xi_3} f_{\S} \Big) \, d \mu_\Phi\\ %= - \sigma^2 (h, \A \Pi f)_{\Hi},
        &=-(3-1)\frac{\sigma^2}{2}(h,\A\Pi f)_{\Hi}.
   \end{aligned}
\end{equation*}

Hence, $c_3=-(3-1)\frac{\sigma^2}{2}$, and Lemma~\ref{auxiliary_lemma_1} yields $c_1=\frac12|c_3|=\frac{\sigma^2}{2}$.~\\

(ii) For the second inequality in \textnormal{\textbf{(H4)}}, we need elliptic regularity, which is available in \cite{DMS15}. Let $h\in \Hi$ be of the form $h= (\I-\P \A^2 \P)f$ for some $f \in \D$, and $g=\Pi_{\S} f$.~From Equation \eqref{2.22_SE(3)}, we obtain 
\begin{equation}
    \begin{aligned}
        \|\A^2 \P f \|_{\Hi} & = \left\|
        v_{3}^{\top} \, \nabla_{\xi}^2 f_{\S} \,  v_{3}  - \frac{1}{2}  \left( (\nabla_{\xi} \Phi)^{\top} \, (I - v_{3} \otimes v_{3}) \, (\nabla_{\xi} f_{\S} )   \right) \right\|_{\Hi}\\
        &\leq \left\| v_{3}^{\top} \, \nabla_{\xi}^2 f_{\S} \,  v_{3} \right\|_{\Hi} 
        + \frac{1}{2}  \left\|
         \left( (\nabla_{\xi} \Phi)^{\top} \, (I - v_{3} \otimes v_{3}) \, (\nabla_{\xi} f_{\S})   \right)
        \right\|_{\Hi}\\
        & \leq \| \, |\nabla_{\xi}^2 f_{\S} | \, \|_{L^2(\gamma_\Phi)} + \frac{1}{2}  \|
          \, |\nabla_{\xi} \Phi|  \, |(\nabla_{\xi} f_{\S})| \, 
        \|_{L^2(\gamma_\Phi)}, 
    \end{aligned}
\end{equation}
where $\Pi f = f_{\S} - (f,1)_{L^2(\gamma_\Phi)}$.~Due to the definition of the operator $(G, \D)$, we have that $\Pi f$ solves the equation
\begin{equation*}
    \Pi f - \frac{1}{3} \left( \Delta_\xi \Pi f - \nabla_{\xi} \Phi \cdot \nabla_{\xi}  \Pi f\right) = \Pi h \quad \text{in } \quad L^2(\gamma_\Phi). 
\end{equation*}
Applying the elliptic a priori estimates, we conclude 
\begin{equation*}
    \begin{aligned}
        \|\A^2 \Pi f\|_{\Hi} \leq c_\Phi \| \Pi h \|_{L^2(\gamma_\Phi)}   \leq c_\Phi  \| h\|_{\Hi}, 
    \end{aligned}
\end{equation*}
for some $c_{\Phi} \in (0, \infty)$.~The constant $c_{\Phi}$ depends only on the constants appearing in the elliptic regularity estimate, namely the Poincaré constant and the bounds on the Hessian of $\Phi$.~Therefore, applying Proposition \ref{auxiliary-prop2.17} we have Assumption \textnormal{\textbf{(H4)}}.
%\end{enumerate}
\end{proof}

\begin{theorem}[Hypocoercivity on $\SE(3)$]\label{theo:hypocoer_SE(3)}
Let $\sigma \in (0, \infty)$, and assume that the potential $\Phi:\mathbb R^3\to\mathbb R$ satisfies conditions \textnormal{\textbf{(C1)}}--\textnormal{\textbf{(C3)}}.~Let $(T_t)_{t\geq 0}$ be the strongly continuous contraction semigroup on $\Hi=L^2(\SE(3),\mu_\Phi)$ generated by the closure of $(\L, \D)$.~Then there exist constants $\kappa_1,\kappa_2 \in (0, \infty)$, depending only on the constants appearing in \textnormal{\textbf{(C1)}}--\textnormal{\textbf{(C3)}} and on $\sigma$, such that for every $f\in \Hi$,
\begin{equation}\label{eq:hypocoercive-exponential-decay}
\left\|T_t f-(f,1)_{\Hi} \right\|_{\Hi}
\leq
\kappa_1 e^{-\kappa_2 t}
\left\|f-(f,1)_{\Hi} \right\|_{\Hi},
\quad t\geq 0.
\end{equation}
\end{theorem}

\begin{proof}
By Theorem \ref{L-is-m-dissipative}, the closure $(L,D(L))$ generates a strongly continuous contraction semigroup $(T_t)_{t\geq0}$ on $\Hi$. Moreover, by Propositions \ref{prop:H1}, \ref{Assumption-H2}, \ref{Macros_coerci}, and \ref{auxiliary_lemma+assumption(H4)_SE(3)}, Assumptions \textnormal{\textbf{(H1)}}--\textnormal{\textbf{(H4)}} of the abstract hypocoercivity theorem are satisfied with $\Lambda_m=\sigma^2, \Lambda_M=\frac{\Lambda}{3}, c_1=\frac{\sigma^2}{2}$ and $c_2=c_{\Phi}$. Hence, the statement follows from Theorem \ref{Hypocoercive_them_for_SE(3)}.
\end{proof}

\begin{remark}\label{rem:final}
\begin{itemize}
    \item[(i)] In Theorem \ref{theo:hypocoer_SE(3)} for each $\kappa_1 > 1$ we can choose  

\begin{equation*}\kappa_2 = \frac{\kappa_1 -1}{\kappa_1} \frac{\sigma^2}{n_1+n_2\sigma^2+n_3\sigma^4},\end{equation*}for suitable constants $n_i>0$, $i=1, \cdots 3$, which depend only on $\Lambda$ and $c_\Phi$. This follows from the quantitative entropy-dissipation argument as in \cite[Section~3.4]{DKMS12}, in the form used in the proof of \cite[Theorem 3.1]{GKMS12}.

\item[(ii)] The assumption $\Phi\in C^\infty(\mathbb{R}^3)$ is only used in the proof of Theorem~\ref{L-is-m-dissipative}, where the essential m-dissipativity of $(\L, \D)$ is established. The proofs of the remaining results require only $\Phi\in C^2(\mathbb{R}^3)$, see \cite[Section 3]{GROT-STIL-2014}. It is therefore natural to expect that the present hypocoercivity result also remains valid under the weaker regularity assumption $\Phi\in C^2(\mathbb{R}^3)$, provided that essential m-dissipativity of $(\L, \D)$ can be established in this setting.
\end{itemize}
\end{remark}
%%%%%%%%%%%%%%%%%%%%%%%%%%%%%%%%%%%%%%%%%%%%%%%%%%%%%%%%%%%%%%%%%5
\medskip
\printbibliography
\addcontentsline{toc}{section}{References}
\AuthorInfo{Martin Grothaus}{Department of Mathematics\\RPTU University Kaiserslautern-Landau, Germany}{grothaus@rptu.de}
\AuthorInfo{Andrea Vanessa Hurtado Quiceno}{Department of Mathematics\\RPTU University Kaiserslautern-Landau, Germany}{avanessa.hquiceno@edu.rptu.de}

\end{document}